\documentclass[11pt,a4paper,reqno]{article}
\usepackage{amsthm, mathrsfs,amssymb,amsmath,}
\usepackage{authblk}
\usepackage{epsfig}
\usepackage{graphicx}
\usepackage{xcolor}
\newtheorem{theorem}{Theorem}[section]
\newtheorem{lemma}[theorem]{Lemma}
\newtheorem{corollary}[theorem]{Corollary}
\theoremstyle{definition}
\newtheorem{definition}[theorem]{Definition}

\newtheorem{proposition}[theorem]{Proposition}

\newtheorem{note}[theorem]{Note}
\theoremstyle{remark}
\newtheorem{remark}[theorem]{Remark}
\usepackage{hyperref}
\allowdisplaybreaks
\numberwithin{equation}{section}

\title{Dimension preserving set-valued approximation and decomposition via metric sum}

\usepackage{authblk}

\author[1]{Ekta Agrawal}
\author[2]{Saurabh Verma}
\affil[1,2]{Department of Applied Sciences, IIIT Allahabad, Prayagraj 211015, India}	
\affil[1]{correspondence to: ekta.agrawal5346@gmail.com}
\affil[2]{saurabhverma@iiita.ac.in}

\begin{document}

\date{}
\maketitle







\begin{abstract}
In the literature, the Minkowski-sum and the metric-sum of compact sets are highlighted. While the first is associative, the latter is not. But the major drawback of the Minkowski combination is that, by increasing the number of summands, this leads to convexification. The present article is uncovered in two folds: The initial segment presents a novel approach to approximate a continuous set-valued function with compact images via a fractal approach using the metric linear combination of sets. The other segment contains the dimension analysis of the distance set of graph of set-valued function and solving the celebrated distance set conjecture. In the end, a decomposition of any continuous convex compact set-valued function is exhibited that preserves the Hausdorff dimension, so this will serve as a method for dealing with complicated set-valued functions.
\end{abstract}
{\bf Keywords:} {Set-valued functions, Metric sum, Constrained approximation, Hausdorff dimension, Box dimension, Assouad dimension, Packing dimension, Iterated function system, Fractal functions.}\\
{\bf Math subject Classification:} {Primary 28A80; Secondary 41A10.}
\maketitle



\section{Introduction}
Set-valued functions (SVFs) find diverse applications in various domains, including mathematical modelling, game theory, differential inclusions, control theory, etc. In recent decades, the approximation and integral of SVFs with compact images have been widely studied and explored, for details, see \cite{AV3,Artstein1,Artstein2,Aumann,Dyn1,Dyn2,Dyn3,Hiai,Michta,Vitale}. Motivated by the real-valued case, Vitale \cite{Vitale} approximated SVFs via Bernstein polynomials, replacing the linear combination of real numbers with the Minkowski combination $+$ of sets. However, the Minkowski combination by increasing the number of summands leads to convexification, that is, the resulting set is convex; for details, see \cite{Dyn2}. As an example, consider the middle third Cantor set $\mathfrak{C}$ on the closed interval $[0,1].$ Then the Steinhaus theorem states that  $$\mathfrak{C}+\mathfrak{C}=[0,2],$$
which is convex (see \cite{Steinhaus}). The generalization of the Steinhaus theorem for the class of fractals determined via the Hutchinson-type operator is studied in \cite{Nikodem}. Consequently, Bernstein polynomial approximation is preferable for the SVFs with convex images. Artstein \cite{Artstein1} interpolated and approximated continuous SVFs with general images by constructing piecewise linear interpolants by introducing special pairs for two compact sets, which is later termed as metric pairs $\Lambda(.,.)$, defined in Sec \ref{sec2}. Afterwards, metric linear combination $\oplus$ of two compact sets is defined by utilizing metric pairs. This linear combination of sets is not necessarily convex; as an example $\mathfrak{C}\oplus \mathfrak{C}=2\mathfrak{C},$ not convex. Now, consider a SVF as $f:[0,1]\rightarrow\mathcal{K}(\mathbb{R})$ as $$f(x):=\{w(x),1\},~ \forall ~x\in [0,1],$$ where $w$ is the real-valued Weierstrass function defined on $[0,1].$ Since the Weierstrass function is a highly irregular non-smooth function, $f$ also possesses an irregular nature. It is obvious to observe that piecewise linear interpolation is profoundly less suitable in this scenario. In order to interpolate and approximate these types of functions, our aim is to introduce fractal interpolation of SVFs, utilizing a metric linear combination of compact sets.\\
The notion of Hausdorff dimension $(\dim_{H}(.))$ lies in the centre of fractal geometry. It is evolved through the notion of Hausdorff measure and is defined for any set or graph of a real or vector-valued function. Mauldin and Williams \cite[Theorem 2]{Mauldin} initiated the decomposition of a real-valued continuous function defined on $[0,1]$ preserving the Hausdorff dimension. With the aim of providing an analogy of this work to SVFs, this article will serve as an initial approach in the direction of dimension analysis and decomposition of continuous SVFs with respect to metric-sum, preserving the Hausdorff dimension. To the best of our knowledge, no prior work in this direction was done earlier. \\
The note is arranged as follows: Sec. \ref{sec2} presents details of the definitions and preliminary results necessary for the work. Secs. \ref{sec3} and \ref{sec4} comprised the main results. 
\section{Preliminaries}\label{sec2}
Let $\mathcal{K}(\mathbb{R})$ denote the collection of all non-empty compact subsets of $\mathbb{R}.$ The Hausdorff distance $\mathfrak{H}$ between any $A,B\in\mathcal{K}(\mathbb{R})$ is elaborated as 
\begin{equation*}
    \mathfrak{H}(A,B):=\max\big\{\max_{a\in A}D(a,B),\max_{b\in B}D(b,A)\big\},
\end{equation*}
where $D(a,B)=\min_{b\in B}|a-b|.$ It is also defined as $$\mathfrak{H}(A,B):=\inf\{\epsilon>0: A\subset B_\epsilon \text{ and } B\subset A_\epsilon\},$$ where $A_\epsilon$ denotes the $\epsilon$-neighbourhood of $A.$
The set of metric pairs of $A$ and $B$ in $\mathcal{K}(\mathbb{R})$ is defined as 
\begin{equation*}
    \Lambda(A,B):=\big\{(a,b)\in A\times B:~ a\in \Lambda_A(b)\text{ or }b\in \Lambda_B(a)\big\},
\end{equation*}
where $\Lambda_A(b):=\{a\in A: D(b,A)=|b-a|\}.$ 
Let $\{A_0,A_1,\ldots,A_N\}\subseteq\mathcal{K}(\mathbb{R}).$ Then the collection of all metric chains is denoted by $\text{Ch}(A_0,A_1,\ldots,A_N)\subseteq A_0\times A_1\times\ldots\times A_N,$ and is defined as
\begin{align*}
     \text{Ch}(A_0&,A_1,\ldots,A_N)\\
     &:=\{(a_0,a_1,\ldots,a_N):(a_j,a_{j+1})\in\Lambda(A_j,A_{j+1}), j=0,1,\ldots, N-1\}.   
\end{align*}
Let $\lambda_0,\lambda_1,\ldots,\lambda_N\in \mathbb{R},$ then the metric linear combination is defined as
\begin{equation*}
    \bigoplus_{j=0}^N\lambda_jA_j:=\Big\{\sum_{j=0}^N\lambda_ja_j:~(a_0,a_1,\ldots,a_N)\in\text{Ch}(A_0,A_1,\ldots,A_N)\Big\}.
\end{equation*}
\begin{note}\cite{Dyn1}\label{note1}
    The basic properties of metric linear combination are as follows:
    \begin{enumerate}
        \item $\bigoplus_{j=0}^N\lambda_jA=\Big(\sum_{j=0}^N\lambda_j\Big)A.$ 
        \item $\bigoplus_{j=0}^N\lambda A_j=\lambda\Big(\bigoplus_{j=0}^NA_j\Big).$
        \item $\bigoplus_{j=0}^N\lambda_jA_j=\bigoplus_{j=0}^N\lambda_{N-j}A_{N-j}.$ 
    \end{enumerate}
\end{note}
In the sequel, the upcoming remark and proposition comment on the associativity of metric linear combination.
\begin{remark}
 It is noted that the Minkowski linear combination of sets is commutative and associative, while metric linear combination of sets is commutative but not associative. As an example, consider $A=\{1,2\}, B=\{7,8,9\}$ and $C=\{-1,-10\}.$ Then $\Lambda(A,B)=\{(1,7),(2,7),(2,8),(2,9)\}, ~\Lambda(B,C)=\{(7,-1),(7,-10),(8,-1),(9,-1)\}.$ So, $\text{Ch}(A,B,C)=\{(1,7,-1),(1,7,-10),(2,7,-1),(2,7,-10),(2,8,-1),(2,9,-1)\}$ and $A\oplus B\oplus C=\{7, -2,8,\\-1,9,10\}.$ Also, $A\oplus B=\{8,9,10,11\}$ and $B\oplus C=\{6,-3,7,8\}.$ It is easily calculated that 
$(A\oplus B)\oplus C=\{7,-2,8,9,10\}$ and $A\oplus (B\oplus C)=\{-2,8,9,10\}.$ Thus, the example illustrates $$(A\oplus B)\oplus C\ne A\oplus (B\oplus C)\ne A\oplus B\oplus C,$$ showing that the metric-sum need not be associative. Again, by definition, observe that $$\text{Ch}(A,B,B)=\{(a,b,b): (a,b)\in \Lambda(A,B)\}.$$
Then $A\oplus B\oplus (-B)=A=A\oplus (B\oplus(-B)).$ But $(A\oplus B)\oplus (-B)\ne A.$ Consider the last example, $A=\{1,2\}, B=\{7,8,9\}$ and $A\oplus B=\{8,9,10,11\}.$ So, $\Lambda(A\oplus B, B)=\{(8,8),(9,9),(10,9),(11,9),(8,7)\}$ and $(A\oplus B)\oplus (-B)=\{0,1,2\}\ne A.$ 
\end{remark}
Let $\mathcal{K}_c(\mathbb{R})$ be the collection of convex compact subsets of $\mathbb{R}.$
\begin{proposition}
    Let $A,B,C\in \mathcal{K}_c(\mathbb{R}).$ Then associativity holds for the metric-sum, that is, 
    $$(A\oplus B)\oplus C= A\oplus (B\oplus C)=A\oplus B\oplus C.$$
\end{proposition}
\begin{proof}
    Let $A=[a,\overline{a}],~ B=[b,\overline{b}]$ and $C=[c,\overline{c}].$ We consider different cases: \\
    \underline{\textbf{Case1:}} $A,B$ and $C$ are mutually disjoint. Now, the following sub-cases are considered separately:
    \begin{itemize}
        \item$a\le \overline{a}<b\le \overline{b}<c\le \overline{c}.$\\
        Then $\Lambda(A,B)=\{(x,b):~x\in A\}\cup\{(\overline{a},x):x\in B\}$ and $\Lambda(B,C)=\{(x,c):~x\in B\}\cup\{(\overline{b},x):x\in C\}.$ This gives $A\oplus B=[a+b,\overline{a}+\overline{b}]$ and $B\oplus C=[b+c,\overline{b}+\overline{c}].$ Also, $$\text{Ch}(A,B,C)=\{(x,b,c): x\in A\}\cup\{(\overline{a},\overline{b},x): x\in C\}\cup\{(\overline{a},x,c):x\in B\}.$$
    Therefore, $A\oplus B\oplus C=[a+b+c,\overline{a}+b+c]\cup[\overline{a}+b+c,\overline{a}+\overline{b}+c]\cup[\overline{a}+\overline{b}+c,\overline{a}+\overline{b}+\overline{c}]=[a+b+c,\overline{a}+\overline{b}+\overline{c}].$ Assume $D=A\oplus B$ and $E=B\oplus C.$ Then 
    \begin{align*}
      &\Lambda(D,C)=\\
      &\begin{cases}
        \{(x,c):x\in D\}\cup\{(\overline{a}+\overline{b},x):x\in C\}; ~~ \text{if }~ \overline{a}+\overline{b}\le c,\\
       \{(x,c):x\in[a+b,c]\}\cup\{(x,x):x\in [c,\overline{a}+\overline{b}]\}\cup\{(\overline{a}+\overline{b},x):&x\in [\overline{a}+\overline{b},\overline{c}]\}; \\ &\text{if }~ \overline{a}+\overline{b}> c,
    \end{cases}  
    \end{align*}
and 
\begin{align*}
   &\Lambda(A,E)=\\
   &\begin{cases}
        \{(x,b+c):x\in A\}\cup\{(\overline{a},x):x\in E\}; ~~ \text{if }~ \overline{a}\le b+c,\\
       \{(x,b+c):x\in[a,b+c]\}\cup\{(x,x):x\in [b+c,\overline{a}]\}\cup\{(\overline{a},x):&x\in [\overline{a},\overline{b}+\overline{c}]\}; \\& \text{if }~ \overline{a}> b+c. 
    \end{cases} 
\end{align*}
As a consequence, we have 
$$(A\oplus B)\oplus C= A\oplus (B\oplus C)=A\oplus B\oplus C.$$
\item $b\le \overline{b}<a\le \overline{a}<c\le \overline{c}.$\\
Then $\Lambda(A,B)=\{(x,\overline{b}):~x\in A\}\cup\{(a,x):x\in B\}$ and $\Lambda(B,C)=\{(x,c):~x\in B\}\cup\{(\overline{b},x):x\in C\}.$ This gives $A\oplus B=[a+b,\overline{a}+\overline{b}]$ and $B\oplus C=[b+c,\overline{b}+\overline{c}].$ Also, $$\text{Ch}(A,B,C)=\{(x,\overline{b},y): x\in A, ~y\in C\}\cup\{(a,x,c):x\in B\}.$$
    Therefore, $A\oplus B\oplus C=[a+b+c,\overline{a}+\overline{b}+\overline{c}].$ Assume $D=A\oplus B$ and $E=B\oplus C.$ Then 
    \begin{align*}
      &\Lambda(D,C)=\\
      &\begin{cases}
        \{(x,c):x\in D\}\cup\{(\overline{a}+\overline{b},x):x\in C\}; ~~ \text{if }~ \overline{a}+\overline{b}\le c,\\
       \{(x,c):x\in[a+b,c]\}\cup\{(x,x):x\in [c,\overline{a}+\overline{b}]\}\cup\{(\overline{a}+\overline{b},x):&x\in [\overline{a}+\overline{b},\overline{c}]\}; \\& \text{if }~ \overline{a}+\overline{b}> c,
    \end{cases}  
    \end{align*}
and 
\begin{align*}
   &\Lambda(A,E)=\\
   &\begin{cases}
        \{(x,b+c):x\in A\}\cup\{(\overline{a},x):x\in E\}; ~~ \text{if }~ \overline{a}\le b+c,\\
       \{(x,b+c):x\in[a,b+c]\}\cup\{(x,x):x\in [b+c,\overline{a}]\}\cup\{(\overline{a},x):&x\in [\overline{a},\overline{b}+\overline{c}]\}; \\& \text{if }~ \overline{a}> b+c,\\
       \{(x,\overline{b}+\overline{c}):x\in A\}\cup\{(a,x):x\in E\}; ~~ \text{if }~ \overline{b}+\overline{c}\le a. 
    \end{cases} 
\end{align*}
As a consequence, we have 
$$(A\oplus B)\oplus C= A\oplus (B\oplus C)=A\oplus B\oplus C.$$
    \end{itemize} 
\underline{\textbf{Case2:}} Suppose $A\cap B\ne \emptyset,~ A\cap C=\emptyset$ and $B\cap C=\emptyset,$ i.e., $a<b<\overline{a}<\overline{b}<c\le\overline{c}.$\\
 Then $\Lambda(A,B)=\{(x,b):~x\in [a,b]\}\cup\{(x,x):x\in[b,\overline{a}]\}\cup\{(\overline{a},x):x\in [\overline{a},\overline{b}]\}$ and $\Lambda(B,C)=\{(x,c):~x\in B\}\cup\{(\overline{b},x):x\in C\}.$ This gives $A\oplus B=[a+b,\overline{a}+\overline{b}]$ and $B\oplus C=[b+c,\overline{b}+\overline{c}].$ Also, $\text{Ch}(A,B,C)=\{(x,b,c): x\in [a,b]\}\cup\{(x,x,c): x\in [b,\overline{a}]\}\cup\{(\overline{a},x,c):x\in [\overline{a},\overline{b}]\}\cup\{(\overline{a},\overline{b},x):x\in C\}.$
    Therefore, $A\oplus B\oplus C=[a+b+c,\overline{a}+\overline{b}+\overline{c}].$ Assume $D=A\oplus B$ and $E=B\oplus C.$ Then 
    \begin{align*}
      &\Lambda(D,C)=\\
      &\begin{cases}
        \{(x,c):x\in D\}\cup\{(\overline{a}+\overline{b},x):x\in C\}; ~~ \text{if }~ \overline{a}+\overline{b}\le c,\\
       \{(x,c):x\in[a+b,c]\}\cup\{(x,x):x\in [c,\overline{a}+\overline{b}]\}\cup\{(\overline{a}+\overline{b},x):&x\in [\overline{a}+\overline{b},\overline{c}]\}; \\
       &\text{if }~ \overline{a}+\overline{b}> c,
    \end{cases}  
    \end{align*}
and 
\begin{align*}
   &\Lambda(A,E)=\\
   &\begin{cases}
        \{(x,b+c):x\in A\}\cup\{(\overline{a},x):x\in E\}; ~~ \text{if }~ \overline{a}\le b+c,\\
       \{(x,b+c):x\in[a,b+c]\}\cup\{(x,x):x\in [b+c,\overline{a}]\}\cup\{(\overline{a},x):&x\in [\overline{a},\overline{b}+\overline{c}]\}; \\&\text{if }~ \overline{a}> b+c. 
    \end{cases} 
\end{align*}
As a consequence, we have 
$$(A\oplus B)\oplus C= A\oplus (B\oplus C)=A\oplus B\oplus C.$$
In a similar manner, the remaining cases are done. This completes the proof.

\end{proof}
\begin{remark}
    Let $A,B$ and $C$ be convex compact subsets of $\mathbb{R}^2.$ Then associativity need not hold in general. Consider $A=[0,1]\times[0,1], B=[10,100]\times[0,4]$ and $C=[90,95]\times[70,80].$ Then $\Lambda(A,B)=\{((a,\overline{a}),(10,\overline{a})):(a,\overline{a})\in A\}\cup\{((1,\overline{b}),(b,\overline{b})): (b,\overline{b})\in B, 0\le\overline{b}\le1\}\cup\{((1,1),(b,\overline{b})):(b,\overline{b})\in B, \overline{b}>1\}$ and $\Lambda(B,C)=\{((c,4),(c,\overline{c})):(c,\overline{c})\in C\}\cup\{((b,\overline{b}),(b,70)):(b,\overline{b})\in B, 90\le b\le 95\}\cup\{((b,\overline{b}),(95,70)):(b,\overline{b})\in B, 95<b\le100\}\cup\{((b,\overline{b}),(90,70)):(b,\overline{b})\in B,10\le b<90\}.$ So, $A\oplus B=([10,11)\times[0,2])\cup([11,101]\times[0,5])$ and $  B\oplus C=([100,180)\times[70,74])\cup([180,190]\times[70,84])\cup((190,195]\times[70,74]).$ Also, $\text{Ch}(A,B,C)=\{((a,\overline{a}),(10,\overline{a}),(90,70)):(a,\overline{a})\in A\}\cup\{((1,\overline{b}),(b,\overline{b}),(90,70)): 10\le b<90, 0\le\overline{b}\le1\}\cup\{((1,\overline{b}),(b,\overline{b}),(b,70)): 90\le b\le95, 0\le\overline{b}\le1\}\cup\{((1,\overline{b}),(b,\overline{b}),(95,70)): 95< b\le 100, 0\le\overline{b}\le1\}\cup\{((1,1),(b,\overline{b}),(90,70)): 10\le b<90,1<\overline{b}\le4\}\cup\{((1,1),(b,\overline{b}),(b,70)): 90\le b\le95,1<\overline{b}\le4\}\cup\{((1,1),(b,\overline{b}),(95,70)):  95< b\le 100,1<\overline{b}\le4\}.$ This gives $$A\oplus B\oplus C=([100,196]\times[70,72])\cup([101,196]\times(72,75]).$$
    Consider $D=A\oplus B.$ Then $\Lambda(D,C)=\{((c,5),(c,\overline{c})):(c,\overline{c})\in C\}\cup\{((d,\overline{d}),(90,70)):(d,\overline{d})\in D,10\le d<90\}\cup\{((d,\overline{d}),(d,70)):(d,\overline{d})\in D,90\le d\le95\}\cup\{((d,\overline{d}),(95,70)):(d,\overline{d})\in D,95< d\le101\}.$ As a result,
    $(A\oplus B)\oplus C=([180,190]\times[70,85])\cup([100,101)\times[70,72])\cup([101,180)\times[70,75])\cup([190,196]\times[70,75]).$
    Clearly, $A\oplus B\oplus C\ne (A\oplus B)\oplus C. $ Thus, associativity does not hold for convex compact sets in higher dimensions.
    \end{remark}
    These observations are utilized consistently throughout the article, and further insights concerning the Hausdorff distance are presented in the upcoming lemmas.
\begin{lemma}\label{11}
Let $A,B,C\in \mathcal{K}(\mathbb{R}).$ Then
$$\mathfrak{H}(A\oplus B, A\oplus C)=\mathfrak{H}(B, C).$$
    \end{lemma}
\begin{proof}
The Hausdorff distance is given as follows:
\begin{align*}
    \mathfrak{H}(&A\oplus B, \ A\oplus C)\\&=\max\Big\{\max_{a+b\in A\oplus B}D(a+b,A\oplus C),\max_{a'+c\in A\oplus C}D(a'+c,A\oplus B)\Big\}.
\end{align*}
   Observing the first term, we get
    \begin{equation*}
        \begin{aligned}
            D(a+b,A\oplus C)&=\min_{a'+c\in A\oplus C}|(a+b)-(a'+c)|\\
            &=\min_{a'+c\in A\oplus C}|(a-a')+(b-c)|.\\
        \end{aligned}
    \end{equation*}
    Clearly, minimum is obtained when $a=a'.$ Also, $(a,b)\in \Lambda(A,B)$ implies either $a\in \Lambda_A(b)$ or $b\in \Lambda_B(a).$ In view of Note \ref{n1}, there exists $(a,c)\in\Lambda(A,C)$ such that
    \begin{align*}
        D(a+b,A\oplus C)
            =\min_{c\in C, a+c\in A\oplus C}|b-c|=D(b,C).
    \end{align*}
    Similarly, $D(a'+c,A\oplus B)=\min_{b\in B, a+b\in A\oplus B}|c-b|=D(c,B).$ Thus,
   \begin{equation*}
        \begin{aligned}
             \mathfrak{H}(A\oplus B, A\oplus C)&=\max\Big\{\max_{a+b\in A\oplus B}D(b,C),\max_{a'+c\in A\oplus C}D(c,B)\Big\}\\
             &=\max\Big\{\max_{b\in B}D(b,C),\max_{c\in C}D(c,B)\Big\}\\
             &=\mathfrak{H}(B,C),
        \end{aligned}
    \end{equation*}
    establishing the assertion.
\end{proof}
\begin{lemma}\label{le1}
Let $A,B,C,D\in\mathcal{K}(\mathbb{R}).$ Then 
$$\mathfrak{H}(A\oplus B, C\oplus D)\le\mathfrak{H}(A, C)+\mathfrak{H}(B,D).$$
\end{lemma}
\begin{proof}
    The proof follows naturally using the triangle inequality.
\end{proof}
\begin{lemma}
    If $A_1,A_2,A_3,B_1,B_2,B_3\in\mathcal{K}_c(\mathbb{R}),$ then 
    $$\mathfrak{H}(A_1\oplus A_2\oplus A_3, B_1\oplus B_2\oplus B_3)\le \mathfrak{H}(A_1,B_1)+\mathfrak{H}(A_2,B_2)+\mathfrak{H}(A_3,B_3).$$
\end{lemma}
\begin{proof}
    As $A_1,A_2,A_3\in\mathcal{K}_c(\mathbb{R}),$ associativity holds, that is,
    $$A_1\oplus A_2\oplus A_3=(A_1\oplus A_2)\oplus A_3=A_1\oplus (A_2\oplus A_3).$$
Consequently, 
\begin{align*}
    &\mathfrak{H}(A_1\oplus A_2\oplus A_3, B_1\oplus B_2\oplus B_3)\\\le&\mathfrak{H}(A_1\oplus A_2\oplus A_3, A_1\oplus B_2\oplus B_3)+\mathfrak{H}(A_1\oplus B_2\oplus B_3, B_1\oplus B_2\oplus B_3)\\
    =&\mathfrak{H}(A_1\oplus (A_2\oplus A_3), A_1\oplus( B_2\oplus B_3))+\mathfrak{H}(A_1\oplus (B_2\oplus B_3), B_1\oplus (B_2\oplus B_3))\\
    =&\mathfrak{H}(A_2\oplus A_3, B_2\oplus B_3)+\mathfrak{H}(A_1, B_1)\\
    \le&\mathfrak{H}(A_2, B_2)+\mathfrak{H}(A_3, B_3)+\mathfrak{H}(A_1, B_1),
\end{align*}
concluding the claim.
\end{proof}
On the same lines, one can conclude the following.
\begin{corollary}\label{cor1}
For any $n\in\mathbb{N}$, assume $A_1,A_2,\ldots,A_n,B_1,B_2,\ldots,B_n\in\mathcal{K}_c(\mathbb{R}),$ then
$$\mathfrak{H}(A_1\oplus A_2\oplus\cdots\oplus A_n, B_1\oplus B_2\oplus\cdots\oplus B_n)\le \mathfrak{H}(A_1,B_1)+\mathfrak{H}(A_2,B_2)+\cdots+\mathfrak{H}(A_n,B_n).$$
\end{corollary}
In view of \cite[Lemma 8]{PSV}, we conclude that 
\begin{lemma}\label{12}
    Let $A,B\in \mathcal{K}(\mathbb{R}),$ and $\alpha,\beta\in\mathbb{R}.$ Then 
    \begin{align*}
      &\mathfrak{H}(\alpha A,\beta A)\le |A||\alpha-\beta|,\\ 
   \text{ and} \hspace{0.5cm}  &\mathfrak{H}(\alpha A,\alpha B)\le |\alpha|\mathfrak{H}(A,B).
    \end{align*}
    where $|A|=\sup_{x,y\in A}|x-y|$ is the diameter of $A$ and $\alpha A=\{\alpha a:~a\in A\}.$
\end{lemma}
Let $+$ denote the Minkowski sum of two sets. It is observed that $$\mathfrak{H}(A\oplus B, C\oplus D)\le\mathfrak{H}(A+ B, C+ D)$$
need not hold for all $A,B,C,D\in \mathcal{K}(\mathbb{R}).$ Here is an example to illustrate the scenario:\\
Consider $A=B=\frac{1}{2}\mathfrak{C},$ where $\mathfrak{C}$ a middle third Cantor set on $[0,1],$ and $C=D=[0,\frac{1}{2}].$ Then $A\oplus B=\mathfrak{C}, ~C\oplus D=[0,1].$ Also, $A+B=[0,1]$ and $C+D=[0,1].$ Clearly, $\mathfrak{H}(A+ B, C+ D)=0,$ and
$$\mathfrak{H}(\mathfrak{C},[0,1])=\max\{\max_{c\in\mathfrak{C}}D(c,[0,1]),\max_{a\in[0,1]}D(a,\mathfrak{C})\}.$$
So, $D(c,[0,1])=0$ for all $c\in\mathfrak{C}$ and
$$D(a,\mathfrak{C})= \left\{\begin{array}{rcl}0, & \mbox{if} \ a\in\mathfrak{C}, \\ a-\frac{1}{3}, & \mbox{if} \ a\in\Big(\frac{1}{3},\frac{1}{2}\Big], \\ \frac{2}{3}-a
, & \mbox{if} \ a\in\Big(\frac{1}{2},\frac{2}{3}\Big),\\ a-\frac{1}{3^2}, & \mbox{if} \ a\in\Big(\frac{1}{3^2},\frac{1}{6}\Big],\\ \frac{2}{3^2}-a, & \mbox{if} \ a\in\Big(\frac{1}{6},\frac{2}{3^2}\Big), \text{ and so on.} \end{array}\right. \ $$
Thus, $\max_{a\in[0,1]}D(a,\mathfrak{C})=\frac{1}{6},$ and $\mathfrak{H}(\mathfrak{C},[0,1])=\frac{1}{6}.$\par

Let us recall some definitions and preliminary statements that will form the foundation of this article. Consider $\mathbb{R}_+$ denote the set of all non-negative real numbers and $(X,\rho)$ be a metric space.
\begin{definition}\cite{Fal,Rogers}
Let $A\subseteq X$ and $r\in \mathbb{R}_+.$ The $r$-dimensional Hausdorff measure of $A$ is defined as $$H^r(A):=\lim_{\delta\to 0^+}H_{\delta}^r(A),$$
where $H_{\delta}^r(A):=\inf\{\sum_{i=1}^{\infty}|U_i|^r:~ A\subseteq\bigcup_{i=1}^\infty U_i, ~|U_i|\le\delta \}$ and $|U_i|=\sup_{x,y\in U_i}\rho(x,y).$ The Hausdorff dimension of $A$ is defined as $$\dim_H(A):=\inf\{r\geq0:~H^r(A)=0\}.$$
\end{definition}
The other fractal dimensions, such as box dimension $(\dim_B(.))$ \cite{Fal}, Assouad dimension $(\dim_A(.))$ \cite{Fraser} and packing dimension $\dim_P(.)$ \cite{Fal} are also defined in the literature. We direct the reader to \cite{Fal,Fraser} for a detailed comparative view.
\begin{definition}\cite{Boyd,Rakotch}
		 A map $h$ on $X$ is said to be a Rakotch contraction if $h$ is a $\phi$-contraction, that is, there exists a map $\phi:\mathbb{R}_+\to \mathbb{R}_+$ such that
  $$\rho(h(x),h(y))\leq \phi(\rho(x,y)),~~\forall~ x,y\in X,$$
 where the function $\frac{\phi(x)}{x}<1$ and decreasing, for all $x>0.$    
	\end{definition}
 \begin{theorem}\cite{Jack,JI}\label{n1}
	  A map $h$ on $X$ is a Rakotch contraction if and only if there exists an increasing function $\phi:\mathbb{R}_+\to \mathbb{R}_+$ such that $h$ is $\phi$-contraction and for all $x>0,$ the function $\frac{\phi(x)}{x}<1$ and decreasing.		
	\end{theorem}
  \begin{theorem}\cite{Rakotch}\label{RK}
		Let $(X,\rho)$ be a complete metric space and $h:X\to X$ be a Rakotch contraction. Then there exists unique $x_*\in X$ such that
  $$h(x_*)=x_*.$$
	\end{theorem}
The previous theorem is called the Rakotch contraction theorem and is the generalization of the Banach contraction theorem (choose $\phi(x)=cx,$ where $0<c<1$). Barnsley \cite{MF1} initiated the seminal approach of interpolating real-valued functions using the Banach contraction theorem, which was later generalized by several authors; for details, see \cite{M2,SR2,Mverma2}. The study of fractals via a dynamical system can be seen in \cite{MBS1,MBS2}.  \par
The upcoming section explores the construction of a set-valued fractal function with an example. To ensure convenience, we denote $\Sigma_N=\{1,2,\ldots,N\}$ and $\Sigma_{0,N}=\{0,1,\ldots,N\}$ throughout the article,  for any $N\in\mathbb{N}.$
\section{Set-valued fractal function}\label{sec3}
Initially, the theory of non-smooth interpolation of the finite data set of $\mathbb{R}^2$ was proposed by Barnsley \cite{MF1} using the concept of an iterated function system (IFS) \cite{Fal,JH}. The interpolant obtained is termed as fractal interpolation function (FIF) and is a fixed point of the Read-Bajraktarevi$\acute{c}$ operator. This concept has a significant impact on the field of interpolation and approximation theory. Ri \cite{SR2} initiated the construction of a real-valued fractal function using the Rakotch fixed point theorem. Primarily, our aim is to demonstrate the set-valued fractal interpolation using the Rakotch contraction principle and metric-sum of sets.\\
Let $I=[a,b]\subseteq\mathbb{R}$ be a compact subset of $\mathbb{R}$ and $N\in\mathbb{N}.$ Suppose $\chi=(x_0,x_1,\ldots,x_N)$ with $a=x_0<x_1<\cdots<x_N=b$ be any partition of $I,$ and the finite data corresponding to $\chi$ is a subset of $I\times\mathcal{K}(\mathbb{R}), $ mentioned as follows:
\begin{equation}\label{data}
    \Big\{(x_n,Y_n)\in I\times\mathcal{K}(\mathbb{R}): n\in \Sigma_{0,N}\Big\}.
\end{equation}
The metric $d$ defined on $I\times \mathcal{K}(\mathbb{R})$ is as follows: 
$$d((x,Y),(x_*,Y_*))=|x-x_*|+\mathfrak{H}(Y,Y_*).$$
From this point on, we elaborate the construction of a set-valued fractal function corresponding to a data set \eqref{data} via the formation of an IFS, see \cite{Fal}, by exerting the seminal approach of Barnsley \cite{MF1}. Divide $I$ into sub-intervals as $I_n=[x_{n-1},x_n],$ for all $n\in \Sigma_N.$ Assume $L_n:I\to I_n$ be a homeomorphism such that 
\begin{subequations}\label{Ln}
 \begin{align}
    & L_n(x_0)=x_{n-1},~L_n(x_N)=x_{n}\text{ and } \label{EQ:A}\\
    &|L_n(x)-L_n(y)|\le\gamma_n|x-y|, ~\forall ~x, y\in  I,
\end{align}   
\end{subequations}
where $0<\gamma_n<1.$ Again, suppose $\Omega_n:I\times \mathcal{K}(\mathbb{R})\to \mathcal{K}(\mathbb{R})$ be a continuous functions such that
\begin{subequations}\label{omega}
 \begin{align}
    &\Omega_n(x_0, Y_0)=Y_{n-1},~\Omega_n(x_N,Y_N)=Y_{n}\text{ and } \label{112}\\
    &\mathfrak{H}\big(\Omega_n(x,Y), \Omega_n(x,Y_*)\big)\le\phi_n\big(\mathfrak{H}(Y,Y_*)\big), 
\end{align}   
\end{subequations}
for all $(x,Y), (x,Y_*)\in I\times \mathcal{K}(\mathbb{R}).$ Also, $\phi_n:\mathbb{R}_+\to \mathbb{R}_+$ is an increasing function with $\frac{\phi_n(t)}{t}<1$ and decreasing, for all $t>0.$ Equation (\ref{112}) is commonly termed as end-point condition. Thus, the required IFS $\mathcal{J}$ consists of mappings on $I\times \mathcal{K}(\mathbb{R})$ as follows:
\begin{equation}\label{IFS}
   \mathcal{J}:=\big\{I\times \mathcal{K}(\mathbb{R}); \ \mathcal{W}_n:~n\in \Sigma_{N}\big\},
\end{equation}
where $\mathcal{W}_n$ is a function on $I\times \mathcal{K}(\mathbb{R})$ as 
$$\mathcal{W}_n(x,Y):=\big(L_n(x),\Omega_n(x,Y)\big), ~\forall ~(x,Y)\in I\times \mathcal{K}(\mathbb{R}).$$
Clearly, $\mathcal{W}_n$ is continuous.
Consider the space $\mathcal{C}(I,\mathcal{K}(\mathbb{R})):=\big\{g:I\to\mathcal{K}(\mathbb{R}): g \text{ is continuous}\big\}$ is a complete metric space equipped with the metric $d_C$ outlined as
\begin{equation*}
    d_C(g,g_*):=\sup_{x\in I}\mathfrak{H}(g(x),g_*(x)).
\end{equation*}
For notational simplicity, $\|g\|_C=d_C(g,\bold{0}),$ where $\bold{0}$ is the zero function. The following result demonstrates the existence of a unique continuous function interpolating data \eqref{data}, and its graph is an attractor of IFS $\mathcal{J}.$
\begin{theorem}\label{th3.1}
Let $\mathcal{W}_n, \ n\in \Sigma_N$ are continuous functions on $I\times\mathcal{K}(\mathbb{R})$ as $\mathcal{W}_n(x,Y)=\big(L_n(x),\Omega_n(x,Y)\big),$ where $L_n$ and $\Omega_n$ are as specified in equations \eqref{Ln} and \eqref{omega}, respectively. Then there exists a unique continuous function $f: I\to\mathcal{K}(\mathbb{R})$ which interpolates data \eqref{data} and satisfies the self-referential equation 
\begin{equation}
     f(x)=\Omega_n\big(L_n^{-1}(x),f(L_n^{-1}(x))\big),
  \end{equation}
  for all $x\in I_n, \ n\in \Sigma_N.$ Also, IFS $\mathcal{J}$ has a unique attractor as
\begin{equation*}
    \mathcal{G}_*(f)=\bigcup_{n\in\Sigma_N}\mathcal{W}_n\big(\mathcal{G}_*(f)\big),
\end{equation*}
where $\mathcal{G}_*(f)=\{(x,f(x))\in I\times\mathcal{K}(\mathbb{R}):x\in I\}.$ Moreover, for a given probability vector $p=(p_n)_{n\in\Sigma_N}$, there exists a unique Borel probability measure $\mu_p$ supported on $\mathcal{G}_*(f)$ such that
\begin{align}\label{e3.1}
\mu_p=\sum_{n \in \Sigma_N} p_n \mu_p \circ \mathcal{W}_n^{-1}.    
\end{align}
\end{theorem}
\begin{proof}
  Define $\mathcal{C}_*(I,\mathcal{K}(\mathbb{R})):=\{g\in\mathcal{C}(I,\mathcal{K}(\mathbb{R})): g(x_0)=Y_0,~g(x_N)=Y_N \}.$ Then $\mathcal{C}_*(I,\mathcal{K}(\mathbb{R}))$ is a closed subset of $\mathcal{C}(I,\mathcal{K}(\mathbb{R})).$ For $h$ be the limit point of $\mathcal{C}_*(I,\mathcal{K}(\mathbb{R})),$ there exists a sequence $\{h_n\}_{n\in\mathbb{N}}\subset\mathcal{C}_*(I,\mathcal{K}(\mathbb{R}))$ such that
\begin{equation*}
    d_C(h_n,h)\to 0 \text{ as }n\to\infty.
\end{equation*}
Thus, for any $\epsilon>0,$ there exists $n_0\in\mathbb{N}$ such that $d_C(h_n,h)<\epsilon,$ for all $n\ge n_0.$ This gives $\mathfrak{H}(h_n(x),h(x))<\epsilon$ implies that $h_n(x)\subseteq (h(x))_\epsilon$ and $h(x)\subseteq (h_n(x))_\epsilon.$ So,
$$Y_0\subseteq (h(x_0))_\epsilon \text{ and }h(x_0)\subseteq (Y_0)_\epsilon.$$
As a consequence $h(x_0)=Y_0.$ Similarly, $h(x_N)=Y_N.$ So, $\mathcal{C}_*(I,\mathcal{K}(\mathbb{R}))$ is complete. Now, define Read-Bajraktarevi$\acute{c}$ (RB) operator $\mathcal{R}:\mathcal{C}_*(I,\mathcal{K}(\mathbb{R}))\to\mathcal{C}_*(I,\mathcal{K}(\mathbb{R}))$ as follows:
  \begin{equation*}
      \mathcal{R}g(x):=\Omega_n\big(L_n^{-1}(x),g(L_n^{-1}(x))\big),
  \end{equation*}
  for all $x\in I_n, \ n\in \Sigma_N.$ Notice that 
  \begin{equation*}
     \mathcal{R}g(x_0)=\Omega_1\big(L_1^{-1}(x_0),g(L_1^{-1}(x_0))\big)=\Omega_1(x_0,g(x_0))=Y_0.
  \end{equation*}
 On the same lines $\mathcal{R}g(x_N)=Y_N,$ claiming that $\mathcal{R}$ is well-defined on $\mathcal{C}_*(I,\mathcal{K}(\mathbb{R})).$ Again, 
  \begin{equation*}
      \begin{aligned}
          d_C\big(\mathcal{R}g,\mathcal{R}g_*\big)&=\sup_{x\in I}\mathfrak{H}\big(\mathcal{R}g(x),\mathcal{R}g_*(x)\big)\\
          &=\max_{x\in I_n, n\in \Sigma_N}\mathfrak{H}\big(\mathcal{R}g(x),\mathcal{R}g_*(x)\big)\\
          &=\max_{x\in I_n, n\in \Sigma_N}\mathfrak{H}\big(\Omega_n(L_n^{-1}(x),g(L_n^{-1}(x))),\Omega_n(L_n^{-1}(x),g_*(L_n^{-1}(x)))\big)\\
          &\le\max_{x\in I_n, n\in \Sigma_N}\phi_n\bigg(\mathfrak{H}\big(g(L_n^{-1}(x)),g_*(L_n^{-1}(x))\big)\bigg).          
      \end{aligned}
  \end{equation*}
  Since $\phi_n$ is an increasing functions on $\mathbb{R}_+,$ we get
  \begin{equation*}
      \begin{aligned}
         d_C\big(\mathcal{R}g,\mathcal{R}g_*\big)&\le\max_{n\in \Sigma_N}\phi_n\bigg(\max_{x\in I_n}\mathfrak{H}\big(g(L_n^{-1}(x)),g_*(L_n^{-1}(x))\big)\bigg)\\
         &\le \max_{n\in \Sigma_N}\phi_n\bigg(\max_{x\in I}\mathfrak{H}\big(g(x),g_*(x)\big)\bigg)\\
         &=\max_{n\in \Sigma_N}\phi_n\big(d_C(g,g_*)\big)\\
         &=\phi\big(d_C(g,g_*)\big),
      \end{aligned}
  \end{equation*}
  where $\phi:\mathbb{R}_+\to\mathbb{R}_+$ is given as $\phi(t)=\max_{n\in \Sigma_N}\phi_n(t).$ It is easily observed that $\phi$ is an increasing function, $\frac{\phi(t)}{t}<1$ and decreasing for all $t>0.$ It follows from Theorem \ref{RK} that $\mathcal{R}$ is a Rakotch contraction on $\mathcal{C}_*(I,\mathcal{K}(\mathbb{R}))$ and there exists a  unique $f\in\mathcal{C}_*(I,\mathcal{K}(\mathbb{R}))$ satisfying $\mathcal{R}(f)=f,$ that is, 
 $$f(x)=\Omega_n\big(L_n^{-1}(x),f(L_n^{-1}(x))\big),$$
  for all $x\in I_n, \ n\in \Sigma_N.$ It is straight forward to verify that $f(x_n)=Y_n,$ for all $n\in \Sigma_{0,N},$ that is, $f$ interpolates initial data set \eqref{data}. In view of the previous equation, we get
  \begin{equation*}
      \begin{aligned}
          \bigcup_{n\in\Sigma_N}\mathcal{W}_n\big(\mathcal{G}_*(f)\big)&=\bigcup_{n\in\Sigma_N}\mathcal{W}_n\big\{(x,f(x)):x\in I\big\}\\
          &=\bigcup_{n\in\Sigma_N}\big\{\big(L_n(x),\Omega_n(x,f(x))\big):x\in I\big\}\\
          &=\bigcup_{n\in\Sigma_N}\big\{\big(L_n(x),f(L_n(x))\big):x\in I\big\}=\mathcal{G}_*(f).         
      \end{aligned}
  \end{equation*}
  Using Theorem  $3.8$  of \cite{Mverma2}, the existence and uniqueness of the measure $\mu_p$ follows. This completes the proof.
\end{proof}
The functions $L_n$ and $\Omega_n,\ n\in \Sigma_N$ can be conveniently chosen as
\begin{subequations}
    \begin{align}
        L_n(x)&=a_nx+b_n,\label{eq:a}\\
        \Omega_n(x,Y)&=Q_n(x,Y)\oplus S_n(x),\label{eq:b}
    \end{align}
\end{subequations}
where $a_n,b_n\in\mathbb{R}$ that are uniquely evaluated using equation \eqref{EQ:A} as 
\begin{equation*}
    a_n=\frac{x_n-x_{n-1}}{x_N-x_0} \text{ and } b_n=\frac{x_nx_0-x_{n-1}x_N}{x_N-x_0}.
\end{equation*}
Also, $Q_n:I\times \mathcal{K}(\mathbb{R})\to \mathcal{K}(\mathbb{R})$ is a continuous and Rakotch contraction corresponding to the second variable, that is,
\begin{equation*}
    \mathfrak{H}\big(Q_n(x,Y),Q_n(x,Y_*)\big)\le \phi_n\big(\mathfrak{H}(Y,Y_*)\big),
\end{equation*}
and $S_n:I\to \mathcal{K}(\mathbb{R})$ are continuous functions satisfying $Q_n(x_0,Y_0)\oplus S_n(x_0) =Y_{n-1}$ and $Q_n(x_N,Y_N)\oplus S_n(x_N) =Y_{n}.$
\begin{corollary}
Let $L_n$ and $\Omega_n,$ for $n\in\Sigma_N$ be functions specified in equations \eqref{eq:a} and \eqref{eq:b}, respectively. Then there exists a unique continuous function $f: I\to\mathcal{K}(\mathbb{R})$ which interpolates data \eqref{data} and satisfies the self-referential equation as 
\begin{equation*}
     f(x)=Q_n(L_n^{-1}(x),f(L_n^{-1}(x)))\oplus S_n(L_n^{-1}(x)),
  \end{equation*}
  for all $x\in I_n, \ n\in \Sigma_N.$ 

\end{corollary}
\begin{figure}[ht]
    \centering
    \includegraphics[width=\textwidth]{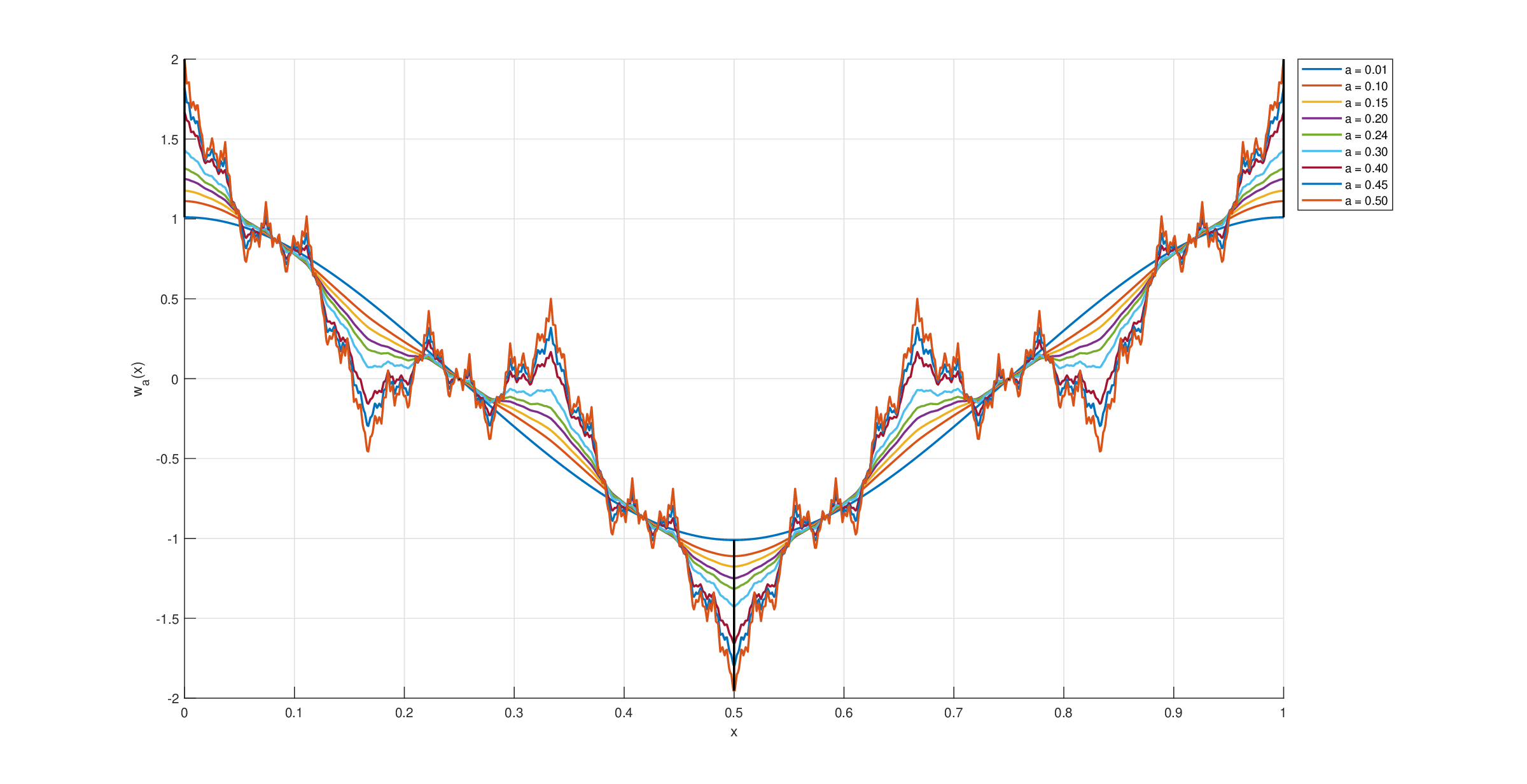}
    \caption{Graphical representation of $W,$ i.e., $\mathcal{G}_*(W).$}
    \label{fig1}
\end{figure}
\subsection{Numerical illustration:}
Let us look at an example so that the construction becomes more visible. First, we will prove the continuity of $h:[0.01,0.5]\rightarrow\mathbb{R}$ as $h(a)=w_a(x),$ for fixed $0\le x\le1,$ where $w_a:[0,1]\rightarrow\mathbb{R}$ is the Weierstrass function given as 
$$w_a(x)=\sum_{k=0}^\infty a^k\cos(2\pi 3^kx), ~~ \forall~ x\in [0,1]. $$
Let $\epsilon>0,$ then there exist $N_1,N_2\in\mathbb{N}$ such that
 $$\sum_{k=l}^\infty a^k_n|\cos(2\pi 3^kx)|\le\frac{\epsilon}{3},~ \forall ~l\ge N_1,\text{ and }\sum_{k=l}^\infty a^k|\cos(2\pi 3^kx)|\le\frac{\epsilon}{3},~ \forall ~l\ge N_2.$$
    Choose $N_3=\max\{N_1,N_2\},$ and $N_4 \in\mathbb{N}$ with $|a_n^k-a^k| \le \frac{\epsilon}{3 N_3} $ for all $n \ge N_4,$ then
\begin{align*}
    |h(a_n)-h(a)|&=|w_{a_n}(x)-w_{a}(x)|\\
    &=\Big|\sum_{k=0}^\infty a_n^k\cos(2\pi 3^kx)-\sum_{k=0}^\infty a^k\cos(2\pi 3^kx)\Big|\\
    &=\Big|\sum_{k=0}^\infty (a_n^k-a^k)\cos(2\pi 3^kx)\Big|\\
    &\le\sum_{k=0}^\infty |a_n^k-a^k||\cos(2\pi 3^kx)|\\
      &\le\sum_{k\le N_3} |a_n^k-a^k|+\sum_{k> N_3} |a_n^k|+\sum_{k> N_3} |a^k|\le \sum_{k\le N_3} \frac{\epsilon}{3 N_3}+\frac{\epsilon}{3}+\frac{\epsilon}{3}=\epsilon.\\
\end{align*}
Thus, $h$ is continuous on $[0.01,0.5].$ Consider the SVF $W:[0,1]\rightarrow\mathcal{K}(\mathbb{R})$ as $W(x)=\{w_a(x):0.01\le a\le 0.5\}.$ Clearly, $W$ is well-defined. To check its continuity, consider $(x_n)$ be a sequence in $[0,1]$ with $x_n\rightarrow x.$
\begin{align*}
    \mathfrak{H}\Big(W(x_n),W(x)\Big)\le
    \sup_{a\in[0.01,0;5]}|w_a(x_n)-w_a(x)|.
\end{align*}
Thus, $W$ is continuous. Its graphical representation is shown in Figure \ref{fig1}, but only a finite number of values of $a$ are plotted to ensure its visibility. Consider the partition $\chi=\Big(0,\frac{1}{4},\frac{1}{2},\frac{3}{4},1\Big)$ of $[0,1],$ then
\begin{align*}
    &W(0)=[1.0101,2]\\
    &W\Big(\frac{1}{4}\Big)=[-0.0067179,-0.00097261]\\
    &W\Big(\frac{1}{2}\Big)=[-1.9532,-1.0101]\\
    &W\Big(\frac{3}{4}\Big)=[-0.0067179,-0.00097261]\\&W(1)=[1.0101,2].
\end{align*}
The metric Bernstein polynomial \cite{Dyn3} is given as
    \begin{align*}
        \mathcal{B}_k^Mf(x)=\bigoplus_{j=0}^k\binom{k}{j}x^j(1-x)^{k-j}f\Big(\frac{j}{k}\Big).
    \end{align*}
For $k=4,$ we get it as follows
\begin{align*}
    \mathcal{B}^M_4W(x)&=\bigoplus_{j=0}^4\binom{4}{j}x^j(1-x)^{4-j}W\Big(\frac{j}{4}\Big)\\&=(1-x)^4W(0)\oplus 4x(1-x)^3W\Big(\frac{1}{4}\Big)\oplus 6x^2(1-x)^2W\Big(\frac{1}{2}\Big)\oplus\\
    &\quad4x^3(1-x)W\Big(\frac{3}{4}\Big)\oplus x^4W(1)\\
    &=\Big((1-x)^4+x^4\Big)W(0)\oplus \Big(4x(1-x)^3+4x^3(1-x)\Big)W\Big(\frac{1}{4}\Big)\oplus 6x^2(1-x)^2W\Big(\frac{1}{2}\Big).
\end{align*}
Now, $\Lambda\Big(W(0),W\big(\frac{1}{4}\big)\Big)=\Big\{(y,-0.00097261):y\in W(0)\Big\}\cup\Big\{(1.0101,y):y\in W\big(\frac{1}{4}\big)\Big\}$ and $\Lambda\Big(W\big(\frac{1}{4}\big),W\big(\frac{1}{2}\big)\Big)=\Big\{(y,-1.0101):y\in W\big(\frac{1}{4}\big)\Big\}\cup\Big\{(-0.0067179,y):y\in W\big(\frac{1}{2}\big)\Big\}.$
Thus, 
\begin{align*}
    &\text{Ch}\Big(W(0),W\Big(\frac{1}{4}\Big),W\Big(\frac{1}{2}\Big)\Big)=\\&\Big\{(y,-0.00097261,-1.0101):y\in W(0)\Big\}\cup\Big\{(1.0101,y,-1.0101):y\in W\Big(\frac{1}{4}\Big)\Big\}\cup\\&\quad\Big\{(1.0101,-0.0067179,y):y\in W\Big(\frac{1}{2}\Big)\Big\}.
\end{align*}
Consequently, 
\begin{align*}
    \mathcal{B}^M_4W(x)&=\Bigg[\Big((1-x)^4+x^4\Big)1.0101-\Big(4x(1-x)^3+4x^3(1-x)\Big)0.00097261\\
    &\quad-6.0606x^2(1-x)^2,\Big((1-x)^4+x^4\Big)2-\Big(4x(1-x)^3+4x^3(1-x)\Big)0.00097261\\
    &\quad-6.0606x^2(1-x)^2\Bigg]\cup\Bigg[\Big((1-x)^4+x^4\Big)1.0101-\Big(4x(1-x)^3+4x^3(1-x)\Big)0.0067179\\
    &\quad-6.0606x^2(1-x)^2,\Big((1-x)^4+x^4\Big)1.0101-\Big(4x(1-x)^3+4x^3(1-x)\Big)0.00097261\\
    &\quad-6.0606x^2(1-x)^2\Bigg]\cup\Bigg[\Big((1-x)^4+x^4\Big)1.0101-\Big(4x(1-x)^3+4x^3(1-x)\Big)0.0067179\\
    &\quad-11.7192x^2(1-x)^2,\Big((1-x)^4+x^4\Big)1.0101-\Big(4x(1-x)^3+4x^3(1-x)\Big)0.0067179\\
    &\quad-6.0606x^2(1-x)^2\Bigg]\\
    &=\Bigg[\Big((1-x)^4+x^4\Big)1.0101-\Big(4x(1-x)^3+4x^3(1-x)\Big)0.0067179\\
    &\quad-11.7192x^2(1-x)^2,\Big((1-x)^4+x^4\Big)2-\Big(4x(1-x)^3+4x^3(1-x)\Big)0.00097261\\
    &\quad-6.0606x^2(1-x)^2\Bigg].
\end{align*}
The graphical representation of $\mathcal{G}_*(\mathcal{B}^M_4W)$ is shown in Figure \ref{fig2}.
\begin{figure}[ht]
    \centering
    \includegraphics[width=\textwidth]{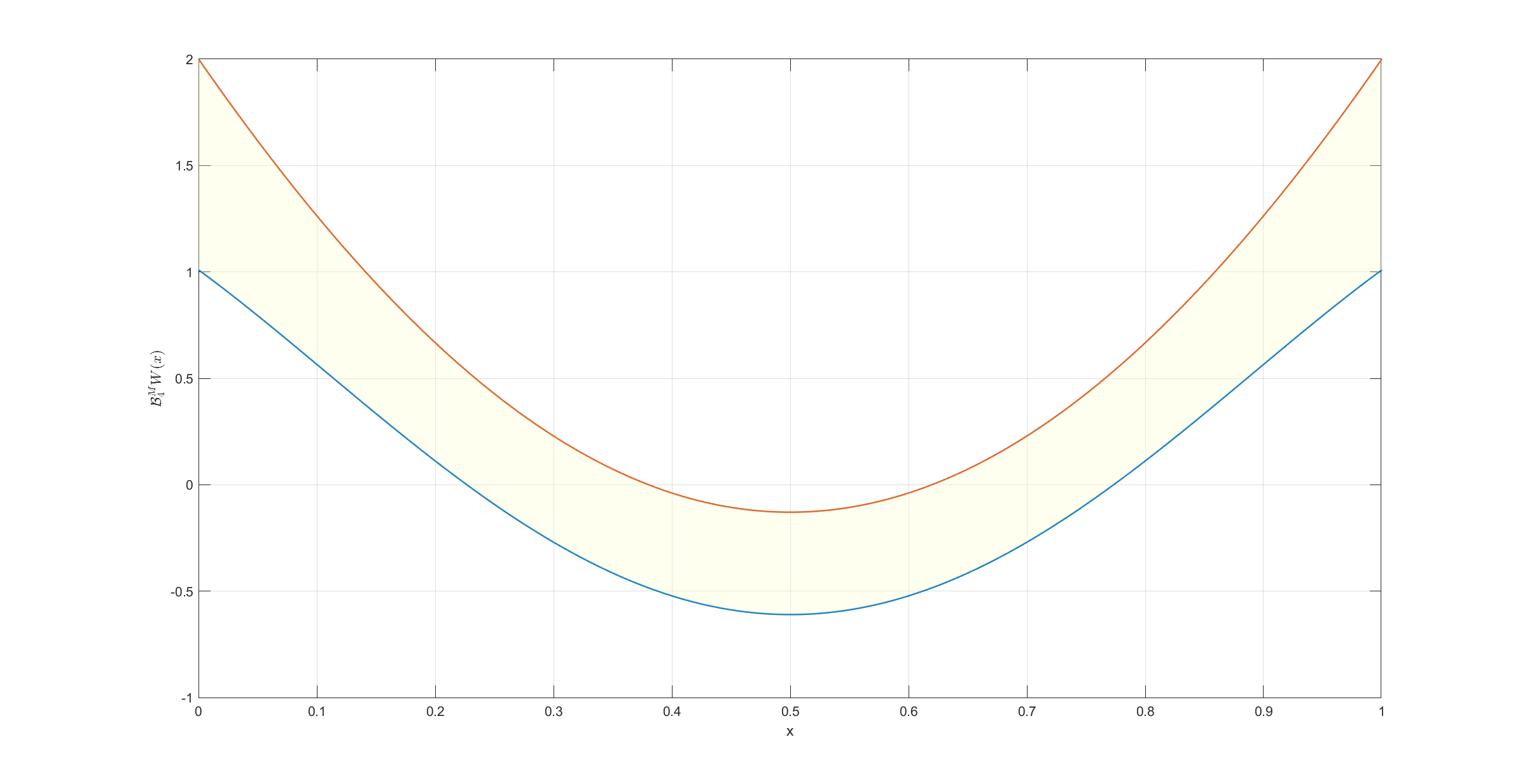}
    \caption{Graphical representation of $\mathcal{B}_4^MW,$ i.e., $\mathcal{G}_*(\mathcal{B}_4^MW$)}
    \label{fig2}
\end{figure}

Our next target is to comment on the mutual singularity of the invariant measure defined with respect to the probability vector in Theorem \ref{th3.1}.
Mor\'an and Rey \cite{Moran} exhibited the singularity of self-similar measure supported on a compact subset of the Euclidean domain with respect to the Hausdorff measure. 
\begin{lemma}\label{homeo}
  Let $p=(p_{n})_{n\in \Sigma_N}$ be a probability vector, and $\mu_p$ and $\nu_p$ be the invariant Borel probability measures generated by the IFSs $\mathcal{J}$ and $\{I ;~L_{n} :n \in \Sigma_N\},$ respectively. Let $P: I  \to \mathcal{G}_*(f)$ be the homeomorphism defined as
 $P(x) = (x,f(x)),$ for all $x \in I  .$ Then \[\nu_p(B) =\mu_p(P(B)),\] for all Borel subsets $B $ of $I .$  
\end{lemma}
 \begin{proof}
 Let $\mathcal{P}(I)$ and $\mathcal{P}(\mathcal{G}_*(f))$ be the spaces of the Borel probability measures supported on $I $ and $\mathcal{G}_*(f)$, respectively. Let us define a homeomorphism $\Psi: \mathcal{P}(\mathcal{G}_*(f))\to \mathcal{P}(I)$ as follows
 \begin{equation*}
   (\Psi\omega)(B)=\omega(P(B)),
 \end{equation*}
 for all $\omega\in \mathcal{P}(\mathcal{G}_*(f))$ and $B$ is a Borel subset of $I.$ Utilizing equation \eqref{e3.1}, estimate
 \begin{align*}
     (\Psi\mu_p)(B)&=\mu_p(P(B))=
     \sum_{n\in\Sigma_N}p_{n}\mu_p \circ \mathcal{W}_{n}^{-1}(P(B))\\&=
     \sum_{n\in\Sigma_N}p_{n}\mu_p \circ (P\circ L_{n}^{-1}(B))=
     \sum_{n\in\Sigma_N}p_{n}(\Psi\mu_p) \circ L_{n}^{-1}(B).
     \end{align*}
Thus, $\Psi\mu_p=\sum_{n\in \Sigma_N}p_{n}(\Psi\mu_p) \circ L_{n}^{-1}.$ But $\nu_p$ is the unique Borel probability measure generated by the IFS $\{I ; ~L_{n}:n \in \Sigma_N\}$ with
 $$\nu_p=\sum_{n\in \Sigma_N}p_{n} \nu_p \circ L_{n}^{-1}.$$
 Therefore, $\Psi\mu_p=\nu_p.$ This completes the assertion.
 \end{proof}
 \begin{theorem}
 Let $\mu_p\text{ and }\mu_{\tilde{p}}$ be the invariant measures generated by the IFS $\mathcal{J}$ corresponding to probability vectors $p=(p_{n})_{n\in \Sigma_N}\text{ and }\tilde{p}=(\tilde{p}_{n})_{n\in \Sigma_N},$ respectively. Then $\mu_p$ is mutually singular to $\mu_{\tilde{p}}.$
\end{theorem}
 \begin{proof}
 Let $\nu_p\text{ and }\nu_{\tilde{p}}$ be the invariant measures generated by IFS $\{I ; \ L_{n}=a_nx+b_n: n \in \Sigma_N\}$ associated to probability vectors $p\text{ and }\tilde{p},$ respectively. Since the IFS consists of similarity mappings, it satisfies open set condition. Exerting \cite[Theorem 2.1]{Moran}, the measures $\nu_p\text{ and }\nu_{\tilde{p}}$ are mutually singular. As a consequence, there exist two disjoint Borel subsets of $I$ (say $B\text{ and }\tilde{B}$) satisfying
$$B\cup \tilde{B}=I~\text{ and }~\nu_p(B)=\nu_{\tilde{p}}(\tilde{B})=0.$$
In view of Lemma \ref{homeo}, $P(B)\text{ and }P(\tilde{B})$ are two Borel subsets of $\mathcal{G}_*(f)$ such that $P(B)\cup P(\tilde{B})=\mathcal{G}_*(f)$ and $P(B)\cap P(\tilde{B})=\emptyset.$ Also $$\mu_p(P(B))=\nu_p(B)=0.$$
Similarly, $\mu_{\tilde{p}}(P(\tilde{B}))=0.$ This completes our proof.
 \end{proof}
We define the space of SVFs of bounded variation as
$$\mathcal{BV}(I,\mathcal{K}(\mathbb{R})):=\{f:I \to \mathcal{K}(\mathbb{R}): V_I(f)< \infty\},$$
where the total variation $V_I(f):= \sup_{\chi} \sum_{i=1}^{l} \mathfrak{H}(f(x_i),f(x_{i-1}))$, and the supremum is taken over all partitions $\chi=(x_0,x_1,\dots,x_l)$ of $I=[a,b]$ with $a=x_0<x_1< \dots<x_l=b.$ Now, we will prove that the space $\mathcal{BV}(I,\mathcal{K}(\mathbb{R}))$ is complete metric space with respect to a suitable metric defined through the metric linear sum.
As a prelude to our next result, we note the following lemma.
\begin{lemma}\label{lem2}
Consider $\{f_k\}_{k\in\mathbb{N}}$ is a sequence of set-valued continuous maps which uniformly converges to $f: I \to \mathcal{K}(\mathbb{R}).$ Then for a given partition $\chi=(x_0,x_1,\ldots,x_l)$ of $I$, we have 
\[ \sum_{i=1}^{l} \mathfrak{H}(f_k(x_i),f_k(x_{i-1})) \to \sum_{i=1}^{l} \mathfrak{H}(f(x_i),f(x_{i-1})).\]
Moreover,
 \(\sup_{\chi}  \sum_{i=1}^{l} \mathfrak{H}(f(x_i),f(x_{i-1})) \le \liminf_{k\to \infty} \sup_{\chi} \sum_{i=1}^{l} \mathfrak{H}(f_k(x_i),f_k(x_{i-1})).\)

\end{lemma}
\begin{theorem}
The space $\Big(\mathcal{BV}(I,\mathcal{K}(\mathbb{R})), d_{BV}\Big)$ is a complete metric space, where $$ d_{BV}(f,g):=d_C(f,g)+  \sup_{\chi}\sum_{i=1}^{l} \mathfrak{H}\Big(  f(x_i) \oplus g (x_{i-1}), g(x_i) \oplus f(x_{i-1})\Big),$$
where $\chi=(x_0,x_1,\ldots,x_l)$ is a partition of $I.$
\end{theorem}
\begin{proof}
Assume that $\{f_k\}$ is a Cauchy sequence in $\mathcal{BV}(I,\mathcal{K}( \mathbb{R}))$ with respect to $d_{BV}.$ Equivalently, for $\epsilon >0$, there exists a $k_0\in \mathbb{N}$ such that 
\[ d_{BV}(f_k ,f_m)  < \epsilon, \text{ for all } k,m \ge k_0.\]
Using the definition of $d_{BV},$ we obtain $ d_C(f_k, f_m)  < \epsilon \text{ for all } k,m \ge k_0.$ Since $(\mathcal{C}(I,\\\mathcal{K}( \mathbb{R})),d_C)$ is a complete metric space, there exists a continuous function $f$ with $ d_C(f_k ,f) \to 0$ as $k \to \infty.$
We claim that $f \in \mathcal{BV}(I,\mathcal{K}( \mathbb{R}))$ and $d_{BV}(f_k ,f)   \to 0$ as $k \to \infty.$ Let $\chi=(x_0,x_1,\ldots,x_l)$ be a partition of $I$ and $k \ge k_0.$ With the reference to Lemma \ref{lem2}, we get
\begin{align*}
&d_{BV}(f_k,f)\\=&d_C(f_k ,f) +\sum_{i=1}^{l} \mathfrak{H}\Big(  f_k(x_i)  \oplus f (x_{i-1}), f(x_i)   \oplus f_k(x_{i-1})\Big) \\
=& \lim_{m \to \infty} \Bigg(d_C(f_k ,f_m) + \sum_{i=1}^{l} \mathfrak{H}\Big(  f_k(x_i)  \oplus f_m (x_{i-1}), f_m(x_i)   \oplus f_k(x_{i-1})\Big)\Bigg)\\  
\le &  \liminf_{m \to \infty} \Bigg(d_C(f_k ,f_m) +\sup_{\chi} \sum_{i=1}^{l} \mathfrak{H}\Big(  f_k(x_i)  \oplus f_m (x_{i-1}), f_m(x_i)   \oplus f_k(x_{i-1})\Big)\Bigg)\\ 
\le &  \sup_{m \ge k_0} \Bigg(d_C(f_k ,f_m) +\sup_{\chi} \sum_{i=1}^{l} \mathfrak{H}\Big(  f_k(x_i)  \oplus f_m (x_{i-1}), f_m(x_i)   \oplus f_k(x_{i-1})\Big)\Bigg)\\
\leq &  \sup_{m \ge k_0} d_{BV}(f_k,f_m)< \epsilon.
 \end{align*}
Since $\chi=(x_0,x_1,\ldots,x_l)$ was arbitrary, we have $d_{BV}(f_k,f) < \epsilon ~ \text{ for all } ~k \ge k_0.$ \\
It remains to show that $f \in \mathcal{BV}(I, \mathcal{K}(\mathbb{R})).$ Now by using Lemma \ref{11}, we have
\begin{align*}
    \sum_{i=1}^{l} \mathfrak{H}(f(x_i),f(x_{i-1}))
   &=\sum_{i=1}^{l} \mathfrak{H}(f(x_i) \oplus f_k(x_{i-1}),f(x_{i-1}) \oplus f_k(x_{i-1}))\\
  &\leq\sum_{i=1}^{l}\mathfrak{H}(f(x_i) \oplus f_k(x_{i-1}),f(x_{i-1}) \oplus f_k(x_{i}))+\\
  &\quad\sum_{i=1}^{l} \mathfrak{H}(f(x_{i-1})\oplus f_k(x_i),f(x_{i-1})  \oplus f_k(x_{i-1}))\\
  &\leq\sum_{i=1}^{l}\mathfrak{H}(f(x_i) \oplus f_k(x_{i-1}),f(x_{i-1}) \oplus f_k(x_{i}))+\\
  &\quad\sum_{i=1}^{l} \mathfrak{H}(f_k(x_i),f_k(x_{i-1}))\\
  &\leq d_{BV}(f_k,f)+ V_I(f_k).
\end{align*}
Since $d_{BV}(f_k,f) < \epsilon$ and $f_k \in \mathcal{BV}(I,\mathcal{K}( \mathbb{R}))$, the above inequality yields that $f \in \mathcal{BV}(I, \mathcal{K}( \mathbb{R})).$
This completes the proof.
\end{proof}
\begin{theorem}
Assume for any $n\in \Sigma_N,$ the functions $S_n \in \mathcal{BV}(I,\mathcal{K}(\mathbb{R}))$ is a compact interval-valued in $I.$ The functions $Q_n:I\times \mathcal{K}(\mathbb{R})\rightarrow\mathcal{K}(\mathbb{R})$ is compact interval-valued continuous and Rakotch contraction in second variable with
\begin{align*}
    \mathfrak{H}&(Q_n(x,Y),Q_n(x_*,Y))\le q_n|x-x_*|\mathfrak{H}(Y,\{0\}),\\
 &Q_n(x,Y)\oplus Q_n(x,Y_*)=Q_n(x,Y\oplus Y_*),
\end{align*}
where $\Big\{(N+1)\frac{\max_{n\in \Sigma_N}\phi_n(t)}{t}+\frac{\max_{n\in \Sigma_N}q_nN|I|}{\min_{n\in \Sigma_
N}|a_n|}\Big\}<1$ for all $t\in\mathbb{R}_+.$ Then, there exists a unique continuous fractal function $h\in \mathcal{BV}(I,\mathcal{K}(\mathbb{R}))$ satisfying  $$Q_n(x,h(x))\oplus S_n(x)=h(L_n(x)),$$ where $x\in I, n\in \Sigma_N.$

\end{theorem}
\begin{proof}
Let $\mathcal{BV}_*(I,\mathcal{K}(\mathbb{R})):=\{g\in\mathcal{BV}(I,\mathcal{K}(\mathbb{R})): g(x_0)=Y_0,~g(x_N)=Y_N \}.$ Then $\mathcal{BV}_*(I,\mathcal{K}(\mathbb{R}))$ is a closed subset of $(\mathcal{BV}(I,\mathcal{K}(\mathbb{R})),d_{BV}),$ and is complete metric space. Define RB operator $\mathcal{R}:\mathcal{BV}_*(I,\mathcal{K}(\mathbb{R}))\to\mathcal{BV}_*(I,\mathcal{K}(\mathbb{R}))$ as follows:
  \begin{equation*}
      \mathcal{R}g(x):=Q_n(L_n^{-1}(x),g(L_n^{-1}(x)))\oplus S_n(L_n^{-1}(x)),
  \end{equation*}
  for all $x\in I_n, \ n\in \Sigma_N.$ Clearly, $\mathcal{R}g(x_n)=Y_n,$ for all $n\in \Sigma_{0,N}.$
 Now, we show that $\mathcal{R}$ is well-defined on $\mathcal{BV}_*(I,\mathcal{K}(\mathbb{R})).$ Assume for $n\in \Sigma_N, ~\chi_n=(x_0,x_1,\ldots,x_l)$ be the partition of the interval $I_n$ with $x_0<x_1<\cdots <x_l.$ Then for any $g\in \mathcal{BV}_*(I,\mathcal{K}(\mathbb{R})),$ consider
 \begin{align*}
     &\sum_{i=1}^l\mathfrak{H}(\mathcal{R}g(x_i),\mathcal{R}g(x_{i-1}))\\
     =&\sum_{i=1}^l\mathfrak{H}\Big(Q_n(L_n^{-1}(x_i),g(L_n^{-1}(x_i)))\oplus S_n(L_n^{-1}(x_i)),Q_n(L_n^{-1}(x_{i-1}),\\
     &\quad g(L_n^{-1}(x_{i-1})))\oplus S_n(L_n^{-1}(x_{i-1}))\Big)\\
     \le&\sum_{i=1}^l\mathfrak{H}\Big(S_n(L_n^{-1}(x_i)),S_n(L_n^{-1}(x_{i-1}))\Big)+\\
     &\quad\sum_{i=1}^l\mathfrak{H}\Big(Q_n(L_n^{-1}(x_i),g(L_n^{-1}(x_i))),Q_n(L_n^{-1}(x_{i-1}),g(L_n^{-1}(x_{i-1})))\Big)\\
     \le& V_I(S_n)+\sum_{i=1}^l\mathfrak{H}\Big(Q_n(L_n^{-1}(x_i),g(L_n^{-1}(x_i))),Q_n(L_n^{-1}(x_{i}),g(L_n^{-1}(x_{i-1})))\Big)+\\
     &\quad\sum_{i=1}^l\mathfrak{H}\Big(Q_n(L_n^{-1}(x_i),g(L_n^{-1}(x_{i-1}))),Q_n(L_n^{-1}(x_{i-1}),g(L_n^{-1}(x_{i-1})))\Big)\\
     \le& \max_{n\in \Sigma_N}V_I(S_n)+\sum_{i=1}^l\phi_n\Big(\mathfrak{H}(g(L_n^{-1}(x_i)),g(L_n^{-1}(x_{i-1})))\Big)+\\
     &\quad\sum_{i=1}^l q_n|L_n^{-1}(x_{i})-L_n^{-1}(x_{i-1})|.\mathfrak{H}(g(L_n^{-1}(x_{i-1})),\{0\})\\
     \le& \max_{n\in \Sigma_N}V_I(S_n)+\sum_{i=1}^l\mathfrak{H}(g(L_n^{-1}(x_i)),g(L_n^{-1}(x_{i-1})))+ \frac{\max_{n\in\Sigma_N}q_n}{\min_{n\in\Sigma_N}|a_n|}|I|.\|g\|_{BV}\\
      \le& \max_{n\in \Sigma_N}V_I(S_n)+V_I(g)+ \frac{\max_{n\in\Sigma_N}q_n}{\min_{n\in\Sigma_N}|a_n|}|I|.\|g\|_{BV}<\infty,\\
 \end{align*}
where $|I|$ is the diameter of $I,$ and $\|g\|_{BV}=d_{BV}(g,\bold{0}),$ ($\bold{0}$ is the zero function). The previous estimation is valid for any partition $\chi_n$ of the sub-interval $I_n,$ for all $n\in \Sigma_N.$ Consequently, 
\begin{align*}
    V_I(\mathcal{R}g)\le N\Big(\max_{n\in \Sigma_N}V_I(S_n)+V_I(g)+ \frac{\max_{n\in\Sigma_N}q_n}{\min_{n\in\Sigma_N}|a_n|}|I|.\|g\|_{BV}\Big)<\infty.
\end{align*}
This claims that $\mathcal{R}$ is well-defined on $\mathcal{BV}_*(I,\mathcal{K}(\mathbb{R})).$ Next, to show $\mathcal{R}$ is a Rakotch contraction on $\mathcal{BV}_*(I,\mathcal{K}(\mathbb{R})).$ Consider $g,g_*\in \mathcal{BV}_*(I,\mathcal{K}(\mathbb{R})),$ in view of Note \ref{note1} and Lemma \ref{le1}, we get
\begin{align*}
     &\sum_{i=1}^l\mathfrak{H}\Big(\mathcal{R}g(x_i)\oplus \mathcal{R}g_*(x_{i-1}),\mathcal{R}g_*(x_{i})\oplus \mathcal{R}g(x_{i-1})\Big)\\
      =&\sum_{i=1}^l\mathfrak{H}\Big(Q_n(L_n^{-1}(x_i),g(L_n^{-1}(x_i)))\oplus S_n(L_n^{-1}(x_i))\oplus Q_n(L_n^{-1}(x_{i-1}),g_*(L_n^{-1}(x_{i-1})))\\
      &\quad \oplus S_n(L_n^{-1}(x_{i-1})),Q_n(L_n^{-1}(x_i),g_*(L_n^{-1}(x_i)))\oplus S_n(L_n^{-1}(x_{i}))\oplus \\
      &\quad Q_n(L_n^{-1}(x_{i-1}),g(L_n^{-1}(x_{i-1})))\oplus S_n(L_n^{-1}(x_{i-1}))\Big).
\end{align*}
In view of associativity and commutativity with respect to metric sum, we get
\begin{align*}
     &\sum_{i=1}^l\mathfrak{H}\Big(\mathcal{R}g(x_i)\oplus \mathcal{R}g_*(x_{i-1}),\mathcal{R}g_*(x_{i})\oplus \mathcal{R}g(x_{i-1})\Big)\\
      =&\sum_{i=1}^l\mathfrak{H}\Big(Q_n(L_n^{-1}(x_i),g(L_n^{-1}(x_i)))\oplus Q_n(L_n^{-1}(x_{i-1}),g_*(L_n^{-1}(x_{i-1}))),\\
      &\quad Q_n(L_n^{-1}(x_{i-1}),g(L_n^{-1}(x_{i-1})))\oplus Q_n(L_n^{-1}(x_i),g_*(L_n^{-1}(x_i)))\Big)\\
      =&\sum_{i=1}^l\mathfrak{H}\Big(Q_n(L_n^{-1}(x_i),g(L_n^{-1}(x_i)))\oplus Q_n(L_n^{-1}(x_{i-1}),g_*(L_n^{-1}(x_{i-1})))\oplus \\
      &\quad (-Q_n(L_n^{-1}(x_i),g_*(L_n^{-1}(x_i)))),Q_n(L_n^{-1}(x_{i-1}),g(L_n^{-1}(x_{i-1})))\Big)\\
       =&\sum_{i=1}^l\mathfrak{H}\Big(Q_n(L_n^{-1}(x_i),(g\oplus(-g_*))(L_n^{-1}(x_i)))\oplus Q_n(L_n^{-1}(x_{i-1}),g_*(L_n^{-1}(x_{i-1}))),\\
      &\quad Q_n(L_n^{-1}(x_{i-1}),g(L_n^{-1}(x_{i-1})))\Big)\\
        =&\sum_{i=1}^l\mathfrak{H}\Big(Q_n(L_n^{-1}(x_i),(g\oplus(-g_*))(L_n^{-1}(x_i))),\\
      &\quad Q_n(L_n^{-1}(x_{i-1}),g(L_n^{-1}(x_{i-1})))\oplus(-Q_n(L_n^{-1}(x_{i-1}),g_*(L_n^{-1}(x_{i-1}))))\Big)\\
  =&\sum_{i=1}^l\mathfrak{H}\Big(Q_n(L_n^{-1}(x_i),(g\oplus(-g_*))(L_n^{-1}(x_i))),\\
      &\quad Q_n(L_n^{-1}(x_{i-1}),(g\oplus(-g_*))(L_n^{-1}(x_{i-1})))\Big)\\
        \le&\sum_{i=1}^l\mathfrak{H}\Big(Q_n(L_n^{-1}(x_i),(g\oplus(-g_*))(L_n^{-1}(x_i))),Q_n(L_n^{-1}(x_{i}),\\
      &\quad (g\oplus(-g_*))(L_n^{-1}(x_{i-1})))\Big)+\sum_{i=1}^l\mathfrak{H}\Big(Q_n(L_n^{-1}(x_i),(g\oplus(-g_*))\\
      &\quad (L_n^{-1}(x_{i-1}))),Q_n(L_n^{-1}(x_{i-1}),(g\oplus(-g_*))(L_n^{-1}(x_{i-1})))\Big)\\\le&\sum_{i=1}^l\phi_n\Bigg(\mathfrak{H}\Big((g\oplus(-g_*))(L_n^{-1}(x_i)),(g\oplus(-g_*))(L_n^{-1}(x_{i-1}))\Big)\Bigg)+\\
      &\quad \sum_{i=1}^lq_n|L_n^{-1}(x_i)-L_n^{-1}(x_{i-1})|\mathfrak{H}\Big((g\oplus(-g_*))(L_n^{-1}(x_{i-1})),\{0\}\Big)\\
       \le&\phi_n\Bigg(\sum_{i=1}^l\mathfrak{H}\Big((g\oplus(-g_*))(L_n^{-1}(x_i)),(g\oplus(-g_*))(L_n^{-1}(x_{i-1}))\Big)\Bigg)+\\
      &\quad \frac{\max_{n\in \Sigma_
N}q_n|I|d_{BV}(g,g_*)}{\min_{n\in \Sigma_
N}|a_n|}\\
\le&\phi_n\Bigg(\sum_{i=1}^l\mathfrak{H}\Big(g(L_n^{-1}(x_i))\oplus g_*(L_n^{-1}(x_{i-1})),g(L_n^{-1}(x_{i-1}))\oplus g_*(L_n^{-1}(x_i))\Big)\Bigg)+\\
      &\quad \frac{\max_{n\in \Sigma_
N}q_n|I|d_{BV}(g,g_*)}{\min_{n\in \Sigma_
N}|a_n|}\\
\le&\max_{n\in \Sigma_
N}\phi_n\Big(d_{BV}(g,g_*)\Big)+\frac{\max_{n\in \Sigma_
N}q_n|I|d_{BV}(g,g_*)}{\min_{n\in \Sigma_
N}|a_n|}.
  \end{align*}
  It is noted that the previous estimation holds for any partition $\chi_n=(x_0,x_1,\ldots,x_l)$ of $I_n,$ where $n\in\Sigma_N.$ Thus,
  \begin{align*}
      &d_{BV}(\mathcal{R}g,\mathcal{R}g_*)\\
=&d_C(\mathcal{R}g,\mathcal{R}g_*)+\sup_{\chi}\sum_{i=1}^n\mathfrak{H}(\mathcal{R}g(x_i)\oplus \mathcal{R}g_*(x_{i-1}),\mathcal{R}g_*(x_{i})\oplus \mathcal{R}g(x_{i-1}))\\
\le&\max_{n\in \Sigma_
N}\phi_n\Big(d_{C}(g,g_*)\Big)+N.\max_{n\in \Sigma_
N}\phi_n\Big(d_{BV}(g,g_*)\Big)+\frac{\max_{n\in \Sigma_
N}q_nN|I|d_{BV}(g,g_*)}{\min_{n\in \Sigma_
N}|a_n|}\\
\le&\max_{n\in \Sigma_
N}\phi_n\Big(d_{BV}(g,g_*)\Big)+N.\max_{n\in \Sigma_
N}\phi_n\Big(d_{BV}(g,g_*)\Big)+\frac{\max_{n\in \Sigma_
N}q_nN|I|d_{BV}(g,g_*)}{\min_{n\in \Sigma_
N}|a_n|}\\
\le&(N+1).\max_{n\in \Sigma_
N}\phi_n\Big(d_{BV}(g,g_*)\Big)+\frac{\max_{n\in \Sigma_
N}q_nN|I|d_{BV}(g,g_*)}{\min_{n\in \Sigma_
N}|a_n|}\\
\le&\Bigg\{(N+1)\frac{\max_{n\in \Sigma_
N}\phi_n\Big(d_{BV}(g,g_*)\Big)}{d_{BV}(g,g_*)}+\frac{\max_{n\in \Sigma_
N}q_nN|I|}{\min_{n\in \Sigma_
N}|a_n|}\Bigg\}d_{BV}(g,g_*).
 \end{align*}
 Now, define $\Theta:\mathbb{R}_+\rightarrow\mathbb{R}_+$ as $\Theta(t):=ct,$ where $$c=(N+1)\frac{\max_{n\in \Sigma_N}\phi_n\Big(d_{BV}(g,g_*)\Big)}{d_{BV}(g,g_*)}+\frac{\max_{n\in \Sigma_N}q_nN|I|}{\min_{n\in \Sigma_
N}|a_n|}.$$
Under the given assumptions, $c<1$ and so $\mathcal{R}$ is a Rakotch contraction on $\mathcal{BV}_*(I,\mathcal{K}(\mathbb{R})).$ Using Theorem \ref{RK}, there exists a unique function $h\in \mathcal{BV}_*(I,\mathcal{K}(\mathbb{R}))$ satisfying $\mathcal{R}h=h,$ concluding the proof. 
\end{proof}

For $0< \sigma \le 1$, the H\"{o}lder space is defined as follows:
\begin{align*}
    \mathcal{HC}^{\sigma}(I, \mathcal{K}( \mathbb{R})  ) := \Big\{&f:I \rightarrow \mathcal{K}( \mathbb{R}): \text{ for some } C_f>0,~\\
    &\mathfrak{H}(f(x),f(y)) \le C_f |x -y|^\sigma,~\forall~x,y \in I\Big\},
\end{align*}
   We now define the following metric using metric linear sum:
\[ d_{{HC}^\sigma}(f,g):= d_C(f,g)+ \sup_{\substack{x,y \in I\\x\ne y}} \frac{\mathfrak{H}\big(f(x) \oplus g(y),g(x)\oplus f(y)\big)}{\lvert x-y \rvert^{\sigma}}.\]
For notation simplicity, $[g]_\sigma=\sup_{\substack{x,y \in I\\x\ne y}} \frac{\mathfrak{H}\big( g(x),g(y)\big)}{\lvert x-y \rvert^{\sigma}}$ and $\|g\|_{{HC}^\sigma}=d_{{HC}^\sigma}(g,\bold{0}),$ ($\bold{0}$ is the zero function.) We aim to show that, with this new metric, $\mathcal{HC}^{\sigma}(I, \mathcal{K}( \mathbb{R})  )$  is complete and not separable. The details about the basic concept of real analysis can be studied in \cite{Rudin}.
\begin{theorem}
The space $(\mathcal{HC}^{\sigma}(I, \mathcal{K}( \mathbb{R})  ), d_{{HC}^\sigma})$ is complete and nonseparable.
\end{theorem}
\begin{proof}
Let $(f_k)$ be a Cauchy sequence in $\mathcal{HC}^{\sigma}(I, \mathcal{K}( \mathbb{R})).$ Thus, for $\epsilon>0,$ there exists $k_0\in\mathbb{N}$ such that
$$d_{{HC}^\sigma}(f_k,f_l)<\epsilon,~~ \text{ for all } k,l\ge k_0.$$ This implies $(f_k)$ is a Cauchy sequence in $(\mathcal{C}(I, \mathcal{K}( \mathbb{R})),d_C),$ which is a complete metric space. Thus, there exists $f\in\mathcal{C}(I, \mathcal{K}( \mathbb{R}))$ with $d_C(f_k,f)\rightarrow0$ as $k\rightarrow\infty.$ Now, our aim is to show that $f\in \mathcal{HC}^{\sigma}(I, \mathcal{K}( \mathbb{R})  )$ and $d_{{HC}^\sigma}(f_k,f)\rightarrow0$ as $k\rightarrow\infty.$ Let $x,y\in I$ with $x\ne y,$ using the fact that every Cauchy sequence is bounded, we notice that
\begin{align*}
    \frac{\mathfrak{H}(f(x),f(y))}{|x-y|^\sigma}&=\lim_{k\to\infty}\frac{\mathfrak{H}(f_k(x),f_k(y))}{|x-y|^\sigma}\\&\le\lim_{k\to\infty}\Big(\sup_{x\in I}\mathfrak{H}(f_k(x),\{0\})+\sup_{\substack{x,y\in I\\x\ne y}}\frac{\mathfrak{H}(f_k(x),f_k(y))}{|x-y|^\sigma}\Big)\\
    &=\lim_{k\to\infty}\|f_k\|<\infty,
\end{align*}
claiming that $f\in\mathcal{HC}^{\sigma}(I, \mathcal{K}( \mathbb{R})). $ Again, 
\begin{align*}
   \frac{\mathfrak{H}(f_k(x)\oplus f(y),f_k(y)\oplus f(x))}{|x-y|^\sigma}
   &=\lim_{l\rightarrow\infty}\frac{\mathfrak{H}(f_k(x)\oplus f_l(y),f_k(y)\oplus f_l(x))}{|x-y|^\sigma}\\
   &\le \limsup_{l\rightarrow\infty}d_{{HC}^\sigma}(f_k,f_l) \le \epsilon,
   \end{align*}
   for every $k\ge k_0$ and the inequality bound is independent from $x \ne y \in I.$ It follows that $f_k$ converges to $f$ in $\mathcal{HC}^{\sigma}(I, \mathcal{K}( \mathbb{R})  )$. Therefore, $\mathcal{HC}^{\sigma}(I, \mathcal{K}( \mathbb{R})  )$ is complete.
Now, for each $t \in (a,b)$, we define the function $f_t: I \to \mathcal{K}( \mathbb{R})$ as 
    
   \[f_t(x)=\begin{cases}
  \{0\}, & \text{if } a\le x < t,\\\\
 \{ |x-t|^{\sigma}\}, & \text{if } t\le  x \le b.
\end{cases}\]
Now, for $t,s \in (a,b)$,  with $t<s$ we have
   \begin{equation*}
       \begin{aligned}
           \sup_{\substack{x,y\in I\\ x\ne y}}\frac{\mathfrak{H}\big(f_t(x) \oplus f_s(y),f_s(x)\oplus f_t(y)\big)}{\lvert x-y \rvert^{\sigma}}& \ge  \frac{\mathfrak{H}(f_t(t) \oplus f_s(s),f_s(t)\oplus f_t(s))}{(s-t)^\sigma}\\ &=\frac{\mathfrak{H}(\{0\} \oplus \{0\},\{0\}\oplus f_t(s))}{(s-t)^\sigma}\\ &=\frac{\mathfrak{H}( \{0\}, \{|s-t|^\sigma\})}{(s-t)^\sigma}= 1.
       \end{aligned}
   \end{equation*}
   Consequently, $d_{{HC}^\sigma}(f_t, f_s) \ge 1$ for distinct $t,s \in (a,b)$ . The collection of such functions is uncountable. Now, the result for nonseparable follows from a standard real-analysis technique.
\end{proof}
We believe that our next theorem perhaps will be beneficial to the researchers working in the area of PDEs and differential inclusions. 
\begin{theorem}
For $0<\sigma_1 < \sigma_2 \le 1$, the inclusion operator $T: (\mathcal{HC}^{\sigma_2}(I, \mathcal{K}( \mathbb{R})), d_{{HC}^{\sigma_2}})\\ \rightarrow (\mathcal{HC}^{\sigma_1}(I, \mathcal{K}( \mathbb{R})  ), d_{{HC}^{\sigma_1}}) $ defined as $T(f)=f$ is continuous, and the image of a bounded set in $\mathcal{HC}^{\sigma_2}(I, \mathcal{K}( \mathbb{R}) )$ under $T$ is relatively compact. Furthermore, $\mathcal{HC}^{\sigma_2}(I, \mathcal{K}( \mathbb{R}))$ is not dense in $\mathcal{HC}^{\sigma_1}(I, \mathcal{K}( \mathbb{R}))$.
\end{theorem}
\begin{proof}
First, note that for $x\ne y$ in $I$, 
\begin{align*}
\frac{\mathfrak{H}(f(x),f(y))}{|x -y|^{\sigma_1}} &= \Big(\frac{\mathfrak{H}(f(x),f(y))}{|x -y|^{\sigma_2}} \Big)^{\sigma_1/ \sigma_2} \mathfrak{H}(f(x),f(y))^{1- \sigma_1/ \sigma_2}\\&\le C_f^{\sigma_1/ \sigma_2} \sup_{x,y\in I}\mathfrak{H}(f(x),f(y))^{1- \sigma_1/ \sigma_2}  .    
\end{align*}
It is clear that $\mathcal{HC}^{\sigma_2}(I, \mathcal{K}( \mathbb{R})) \subseteq \mathcal{HC}^{\sigma_1}(I, \mathcal{K}( \mathbb{R}))$ and hence the inclusion map $T$ is well-defined. Let $(f_k)_{k\in \mathbb{N}}$ be a sequence in $(\mathcal{HC}^{\sigma_2}(I, \mathcal{K}( \mathbb{R})),d_{{HC}^{\sigma_2}})$ converging to $f.$ Observe that
\begin{align*}
    &d_{{HC}^{\sigma_1}}(T(f_k), T(f))\\
    =&d_{{HC}^{\sigma_1}}(f_k, f)\\
    =&d_C(f_k,f)+\sup_{\substack{x,y\in I\\x\ne y}}\frac{\mathfrak{H}(f_k(x)\oplus f(y), f(x)\oplus f_k(y))}{|x-y|^{\sigma_1}}\\
    =&d_C(f_k,f)+\sup_{\substack{x,y\in I\\x\ne y}}\frac{\mathfrak{H}(f_k(x)\oplus f(y), f(x)\oplus f_k(y))}{|x-y|^{\sigma_2}}|x-y|^{\sigma_2-\sigma_1}\\
    \le& \max\{ \sup_{x,y\in I}|x-y|^{\sigma_2-\sigma_1},1\}\Big(d_C(f_k,f)+\sup_{\substack{x,y\in I\\x\ne y}}\frac{\mathfrak{H}(f_k(x)\oplus f(y), f(x)\oplus f_k(y))}{|x-y|^{\sigma_2}}\Big)
    \\
    =& \max\{(b-a)^{\sigma_2-\sigma_1},1\}d_{{HC}^{\sigma_2}}(f_k,f).
\end{align*}
This implies that $d_{{HC}^{\sigma_1}}(T(f_k), T(f))\rightarrow0$ as $k\rightarrow\infty,$ showing that $T$ is continuous. Let $B$ be any bounded set in $\mathcal{HC}^{\sigma_2}(I, \mathcal{K}( \mathbb{R})).$ It is remaining to show that  $T(B)$ is relatively compact (that is, its closure is compact) in $\mathcal{HC}^{\sigma_1}(I, \mathcal{K}( \mathbb{R})).$ \\
Let $(f_k)_{k\in\mathbb{N}}$ be a bounded sequence in $\mathcal{HC}^{\sigma_2}(I, \mathcal{K}( \mathbb{R})),$ that is, there is $M\ge 0$ such that 
$$\|f_k\|_{{HC}^{\sigma_2}}=d_{{HC}^{\sigma_2}}(f_k,\bold{0})\le M, ~\text{ for all }k\in\mathbb{N}.$$
Consider
\begin{align*}
    \|Tf_k\|_{{HC}^{\sigma_1}}&=d_{{HC}^{\sigma_1}}(Tf_k,\bold{0})\\&=d_{{HC}^{\sigma_1}}(f_k,\bold{0})\\
    &=d_C(f_k,\bold{0})+\sup_{\substack{x,y\in I\\x\ne y}}\frac{\mathfrak{H}(f_k(x), f_k(y))}{|x-y|^{\sigma_1}}\\
    &=d_C(f_k,\bold{0})+\sup_{\substack{x,y\in I\\x\ne y}}\frac{\mathfrak{H}(f_k(x), f_k(y))}{|x-y|^{\sigma_2}}|x-y|^{\sigma_2-\sigma_1}\\
    &\le\max\{(b-a)^{\sigma_2-\sigma_1},1\}d_{{HC}^{\sigma_2}}(f_k,\bold{0})\le M \max\{(b-a)^{\sigma_2-\sigma_1},1\}.
\end{align*}
Let $x,y\in I,$ then
\begin{align*}
    \mathfrak{H}(Tf_k(x),Tf_k(y))=\mathfrak{H}(f_k(x),f_k(y))
    &=\frac{\mathfrak{H}(f_k(x),f_k(y))}{|x-y|^{\sigma_2}}|x-y|^{\sigma_2}\\&\le |x-y|^{\sigma_2}d_{{HC}^{\sigma_2}}(f_k,\bold{0}).
\end{align*}
Thus, for any $\epsilon>0,$ there is $\delta=(\frac{\epsilon}{M})^{-\sigma_2}$ such that 
$$|x-y|<\delta~\text{ implies }~\mathfrak{H}(Tf_k(x),Tf_k(y))<\epsilon.$$ 
As a result, $(Tf_k)_{k\in\mathbb{N}}$ is uniformly equicontinuous on $I.$ In view of the Arzel$\grave{a}$-Ascoli theorem, there exists a convergence subsequence of $(Tf_k)_{k\in\mathbb{N}}$ (denoted as $(Tf_{k_j})_{j\in\mathbb{N}}$) that converges uniformly with respect to metric $d_{C}$. Again,
\begin{align*}
&\frac{\mathfrak{H}(Tf_{k_j}(x)\oplus Tf_{k_l}(y),Tf_{k_l}(x)\oplus Tf_{k_j}(y))}{|x -y|^{\sigma_1}} \\=&\frac{\mathfrak{H}(f_{k_j}(x)\oplus f_{k_l}(y),f_{k_l}(x)\oplus f_{k_j}(y))}{|x -y|^{\sigma_2}}.|x-y|^{\sigma_2-\sigma_1}\\
\le&(b-a)^{\sigma_2-\sigma_1}d_{{HC}^{\sigma_2}}(f_{k_j},f_{k_l}).
\end{align*}
claiming that the sequence $(Tf_k)_{k\in\mathbb{N}}$ has a Cauchy subsequence in $T(B),$ so $T(B)$ is relatively compact. Our next target is to show that $\mathcal{HC}^{\sigma_2}(I, \mathcal{K}( \mathbb{R}))$ is not dense in $\mathcal{HC}^{\sigma_1}(I, \mathcal{K}( \mathbb{R})).$ On the contrary, suppose $\overline{\mathcal{HC}^{\sigma_2}(I, \mathcal{K}( \mathbb{R}))}=\mathcal{HC}^{\sigma_1}(I, \mathcal{K}( \mathbb{R})).$ Also, $\mathcal{HC}^{\sigma_2}(I, \mathcal{K}( \mathbb{R}))=\bigcup_{n\in\mathbb{N}}B_n(\bold{0}),$ where $B_n(\bold{0})=\{f\in\mathcal{HC}^{\sigma_2}(I, \mathcal{K}( \mathbb{R})):d_{{HC}^{\sigma_2}}(f,\bold{0})<n\}.$ So,$$\mathcal{HC}^{\sigma_1}(I, \mathcal{K}( \mathbb{R}))=\overline{\bigcup_{n\in\mathbb{N}}B_n(\bold{0})}=\overline{\bigcup_{n\in\mathbb{N}}\overline{B_n(\bold{0})}}.$$ 
Since $\overline{B_n(\bold{0})}$ is compact in $\mathcal{HC}^{\sigma_1}(I, \mathcal{K}( \mathbb{R}))$, it is separable, and this implies that $\bigcup_{n\in\mathbb{N}}\overline{B_n(\bold{0})}$ is separable, that is, $\mathcal{HC}^{\sigma_1}(I, \mathcal{K}( \mathbb{R}))$ is separable, which is a contradiction.
\end{proof}

Next, we show the existence of a class of set-valued fractal functions in the H\"older spaces.
\begin{theorem}\label{holder1}
Assume for any $n\in \Sigma_N,$ the functions $S_n \in \mathcal{HC}^\sigma(I,\mathcal{K}(\mathbb{R}))$ are a compact interval-valued in $I.$ The functions $Q_n:I\times \mathcal{K}(\mathbb{R})\rightarrow\mathcal{K}(\mathbb{R})$ are compact interval-valued continuous and Rakotch contraction in second variable with $\phi_n(\beta t)=\beta\phi_n(t),$ for all $\beta,t>0$ and
\begin{align*}
    \mathfrak{H}&(Q_n(x,Y),Q_n(x_*,Y))\le q_n|x-x_*|^\sigma\mathfrak{H}(Y,\{0\}),\\
 &Q_n(x,Y)\oplus Q_n(x,Y_*)=Q_n(x,Y\oplus Y_*),
\end{align*}
where 
$$\Bigg\{\Big(1+\frac{N}{\min_{n\in \Sigma_N}|a_n|^\sigma}\Big)\frac{\max_{n\in \Sigma_N}\phi_n(t)}{t}+ N\frac{\max_{n\in \Sigma_N}q_n}{\min_{n\in \Sigma_N}|a_n|^\sigma}\Bigg\}<1$$ for all $t>0.$ Then, there exists a unique continuous fractal function $h\in \mathcal{HC}(I,\mathcal{K}(\mathbb{R}))$ satisfying  $$Q_n(x,h(x))\oplus S_n(x)=h(L_n(x)),$$ where $x\in I, n\in \Sigma_N.$
\end{theorem}
\begin{proof}
Let us define a set $\mathcal{HC}^\sigma_*(I,\mathcal{K}(\mathbb{R})):=\{g\in \mathcal{HC}^\sigma(I,\mathcal{K}(\mathbb{R})): g(x_0)=Y_0,g(x_N)=Y_N\}.$ Then $\mathcal{HC}^\sigma_*(I,\mathcal{K}(\mathbb{R}))$ is a closed subset of $\mathcal{HC}^\sigma(I,\mathcal{K}(\mathbb{R})),$ and is a complete metric space with respect to $d_{HC}.$ The RB operator $\mathcal{R}:\mathcal{HC}^\sigma_*(I,\mathcal{K}(\mathbb{R}))\rightarrow\mathcal{HC}^\sigma_*(I,\mathcal{K}(\mathbb{R}))$ is defined as
 \begin{equation*}
      \mathcal{R}g(x):=Q_n(L_n^{-1}(x),g(L_n^{-1}(x)))\oplus S_n(L_n^{-1}(x)),
  \end{equation*}
  for all $x\in I_n, \ n\in \Sigma_N.$ Clearly, $\mathcal{R}g(x_n)=Y_n,$ for all $n\in \Sigma_{0,N}.$ First, we aim to show that $\mathcal{R}$ is well defined in $\mathcal{HC}^\sigma_*(I,\mathcal{K}(\mathbb{R})).$
  \begin{align*}
     [\mathcal{R}g]_\sigma= &\sup_{\substack{x,y\in I\\x\ne y}}\frac{\mathfrak{H}(\mathcal{R}g(x),\mathcal{R}g(y))}{|x-y|^\sigma}\\
      \le& N\max_{n\in \Sigma_N}\sup_{\substack{x,y\in I_n\\x\ne y}}\frac{\mathfrak{H}(\mathcal{R}g(x),\mathcal{R}g(y))}{|x-y|^\sigma}\\
    \le&N\max_{n\in \Sigma_N}\Bigg\{\sup_{\substack{x,y\in I_n\\x\ne y}}\frac{\mathfrak{H}(Q_n(L_n^{-1}(x),g(L_n^{-1}(x))),Q_n(L_n^{-1}(y),g(L_n^{-1}(y))))}{|x-y|^\sigma}+\\
    &\quad \sup_{\substack{x,y\in I_n\\x\ne y}}\frac{\mathfrak{H}(S_n(L_n^{-1}(x)),S_n(L_n^{-1}(y)))} {|x-y|^\sigma} \Bigg\}\\
    \le&N\max_{n\in \Sigma_N}\Bigg\{\sup_{\substack{x,y\in I_n\\x\ne y}}\frac{\mathfrak{H}(Q_n(L_n^{-1}(x),g(L_n^{-1}(x))),Q_n(L_n^{-1}(x),g(L_n^{-1}(y))))}{|x-y|^\sigma}+\\
    &\quad \sup_{\substack{x,y\in I_n\\x\ne y}}\frac{\mathfrak{H}(Q_n(L_n^{-1}(x),g(L_n^{-1}(y))),Q_n(L_n^{-1}(y),g(L_n^{-1}(y))))}{|x-y|^\sigma}+\frac{[S_n]_\sigma}{\min_{n\in \Sigma_N}|a_n|^\sigma}\Bigg\}\\
     \le&N\max_{n\in \Sigma_N}\Bigg\{\sup_{\substack{x,y\in I_n\\x\ne y}}\frac{\phi_n\Big(\mathfrak{H}(g(L_n^{-1}(x)),g(L_n^{-1}(y)))\Big)}{|x-y|^\sigma}+\\
      &\quad \sup_{\substack{x,y\in I_n\\x\ne y}}\frac{q_n|L_n^{-1}(x)-L_n^{-1}(y)|^\sigma.\|g\|_C}{|x-y|^\sigma}+\frac{[S_n]_\sigma}{\min_{n\in \Sigma_N}|a_n|^\sigma}\Bigg\}\\
       \le&N\max_{n\in \Sigma_N}\Bigg\{\sup_{\substack{x,y\in I_n\\x\ne y}}\frac{\mathfrak{H}(g(L_n^{-1}(x)),g(L_n^{-1}(y)))}{|x-y|^\sigma}+\frac{(q_n.\|g\|_C)+[S_n]_\sigma}{\min_{n\in \Sigma_N}|a_n|^\sigma}\Bigg\}\\
       \le&N\Bigg\{\frac{[g]_\sigma+(\max_{n\in \Sigma_N}q_n.\|g\|_C)+\max_{n\in \Sigma_N}[S_n]_\sigma}{\min_{n\in \Sigma_N}|a_n|^\sigma}\Bigg\}.
  \end{align*}
  Thus, $\mathcal{R}$ is well-defined on $\mathcal{HC}^\sigma_*(I,\mathcal{K}(\mathbb{R})).$ Again to show $\mathcal{R}$ is a Rakotch contraction on $\mathcal{HC}^\sigma_*(I,\mathcal{K}(\mathbb{R})).$ Let $g,g_*\in\mathcal{HC}^\sigma_*(I,\mathcal{K}(\mathbb{R})),$ then
  \begin{align*}
      &\sup_{\substack{x,y\in I\\x\ne y}}\frac{\mathfrak{H}(\mathcal{R}g(x)\oplus \mathcal{R}g_*(y),\mathcal{R}g_*(x)\oplus\mathcal{R}g(y))}{|x-y|^\sigma}\\
   \le& N\max_{n\in \Sigma_N}\Bigg\{\sup_{\substack{x,y\in I_n\\x\ne y}}\frac{\mathfrak{H}\Big(\substack{Q_n(L_n^{-1}(x),g(L_n^{-1}(x)))\oplus S_n(L_n^{-1}(x))\oplus Q_n(L_n^{-1}(y),g_*(L_n^{-1}(y)))\\\oplus S_n(L_n^{-1}(y)),Q_n(L_n^{-1}(x),g_*(L_n^{-1}(x)))\oplus S_n(L_n^{-1}(x))\\\oplus Q_n(L_n^{-1}(y),g(L_n^{-1}(y)))\oplus S_n(L_n^{-1}(y))}\Big)}{|x-y|^\sigma}\Bigg\}\\ 
   =& N\max_{n\in \Sigma_N}\Bigg\{\sup_{\substack{x,y\in I_n\\x\ne y}}\frac{\mathfrak{H}\Big(\substack{Q_n(L_n^{-1}(x),g(L_n^{-1}(x)))\oplus Q_n(L_n^{-1}(y),g_*(L_n^{-1}(y))),\\Q_n(L_n^{-1}(x),g_*(L_n^{-1}(x)))\oplus Q_n(L_n^{-1}(y),g(L_n^{-1}(y)))}\Big)}{|x-y|^\sigma}\Bigg\}\\ 
   =& N\max_{n\in \Sigma_N}\Bigg\{\sup_{\substack{x,y\in I_n\\x\ne y}}\frac{\mathfrak{H}\Big(\substack{Q_n(L_n^{-1}(x),(g\oplus(-g_*))(L_n^{-1}(x))),Q_n(L_n^{-1}(y),(g\oplus(-g_*))(L_n^{-1}(y)))}\Big)}{|x-y|^\sigma}\Bigg\}\\ 
    \le& N\max_{n\in \Sigma_N}\sup_{\substack{x,y\in I_n\\x\ne y}}\Bigg\{\frac{\mathfrak{H}\Big(\substack{Q_n(L_n^{-1}(x),(g\oplus(-g_*))(L_n^{-1}(x))),Q_n(L_n^{-1}(x),(g\oplus(-g_*))(L_n^{-1}(y)))}\Big)}{|x-y|^\sigma}\\ 
    &\quad +\frac{\mathfrak{H}\Big(\substack{Q_n(L_n^{-1}(x),(g\oplus(-g_*))(L_n^{-1}(y))),Q_n(L_n^{-1}(y),(g\oplus(-g_*))(L_n^{-1}(y)))}\Big)}{|x-y|^\sigma}\Bigg\}\\
    \le& N\max_{n\in \Sigma_N}\sup_{\substack{x,y\in I_n\\x\ne y}}\Bigg\{\frac{\phi_n\Big(\mathfrak{H}\Big((g\oplus(-g_*))(L_n^{-1}(x)),(g\oplus(-g_*))(L_n^{-1}(y))\Big)\Big)}{|x-y|^\sigma}\\ 
    &\quad +\frac{q_n|L_n^{-1}(x)-L_n^{-1}(y)|^\sigma\mathfrak{H}\Big((g\oplus(-g_*))(L_n^{-1}(y)),\{0\}\Big)}{|x-y|^\sigma}\Bigg\}\\
     \le& N\max_{n\in \Sigma_N}\sup_{\substack{x,y\in I_n\\x\ne y}}\phi_n\Bigg(\frac{\mathfrak{H}\Big(g(L_n^{-1}(x))\oplus g_*(L_n^{-1}(y)),g_*(L_n^{-1}(x))\oplus g(L_n^{-1}(y))\Big)}{|x-y|^\sigma}\Bigg)\\ 
    &\quad +N\frac{\max_{n\in \Sigma_N}q_n}{\min_{n\in \Sigma_N}|a_n|^\sigma}d_C(g,g_*)\\
    \le& N\max_{n\in \Sigma_N}\phi_n\Bigg(\sup_{\substack{x,y\in I_n\\x\ne y}}\frac{\mathfrak{H}\Big(g(L_n^{-1}(x))\oplus g_*(L_n^{-1}(y)),g_*(L_n^{-1}(x))\oplus g(L_n^{-1}(y))\Big)}{|x-y|^\sigma}\Bigg)\\ 
    &\quad +N\frac{\max_{n\in \Sigma_N}q_n}{\min_{n\in \Sigma_N}|a_n|^\sigma}d_C(g,g_*)\\
     =& N\frac{\max_{n\in \Sigma_N}\phi_n(d_{{HC}^\sigma}(g,g_*))}{\min_{n\in \Sigma_N}|a_n|^\sigma}+N\frac{\max_{n\in \Sigma_N}q_n}{\min_{n\in \Sigma_N}|a_n|^\sigma}d_{{HC}^\sigma}(g,g_*).
  \end{align*}
In view of last estimation, notice that
\begin{align*}
    d_{{HC}^\sigma}(\mathcal{R}g,\mathcal{R}g_*)&=d_C(\mathcal{R}g,\mathcal{R}g_*)+\sup_{\substack{x,y\in I\\x\ne y}}\frac{\mathfrak{H}(\mathcal{R}g(x)\oplus \mathcal{R}g_*(y),\mathcal{R}g_*(x)\oplus\mathcal{R}g(y))}{|x-y|^\sigma}\\
    &\le \max_{n\in \Sigma_
N}\phi_n\Big(d_{{HC}^\sigma}(g,g_*)\Big)+N\frac{\max_{n\in \Sigma_N}\phi_n(d_{{HC}^\sigma}(g,g_*))}{\min_{n\in \Sigma_N}|a_n|^\sigma}+\\
&\quad N\frac{\max_{n\in \Sigma_N}q_n}{\min_{n\in \Sigma_N}|a_n|^\sigma}d_{{HC}^\sigma}(g,g_*)\\
 &\le \Bigg\{\Big(1+\frac{N}{\min_{n\in \Sigma_N}|a_n|^\sigma}\Big)\frac{\max_{n\in \Sigma_N}\phi_n(d_{{HC}^\sigma}(g,g_*))}{d_{{HC}^\sigma}(g,g_*)}+\\
&\quad N\frac{\max_{n\in \Sigma_N}q_n}{\min_{n\in \Sigma_N}|a_n|^\sigma}\Bigg\}d_{{HC}^\sigma}(g,g_*).
\end{align*}
Assume the function $\Theta:\mathbb{R}_+\rightarrow\mathbb{R}_+$ as $\Theta(t):=ct,$ where $$c=\Big(1+\frac{N}{\min_{n\in \Sigma_N}|a_n|^\sigma}\Big)\frac{\max_{n\in \Sigma_N}\phi_n(d_{{HC}^\sigma}(g,g_*))}{d_{HC}(g,g_*)}+ N\frac{\max_{n\in \Sigma_N}q_n}{\min_{n\in \Sigma_N}|a_n|^\sigma}.$$
Thus, under given assumption, we get $\Theta$ as an increasing function with $\frac{\Theta(t)}{t}<1$ and decreasing for all $t\in\mathbb{R}_+.$ So, $\mathcal{R}$ is a Rakotch contraction on $\mathcal{HC}_*^\sigma(I,\mathcal{K}(\mathbb{R})).$ With reference to Theorem \ref{RK}, there exists unique $h\in \mathcal{HC}_*^\sigma(I,\mathcal{K}(\mathbb{R}))$ such that $\mathcal{R}h=h,$ concluding the proof. 
\end{proof}
We recall that every Lipschitz function is H\"older continuous with exponent $1.$ 
The forthcoming theorem demonstrates the existence of a probability vector so that the associated invariant measure is equivalent to the Hausdorff measure. 
\begin{theorem}\label{absolutecontinuity}
Let $\mathcal{J}$ be the IFS consists of mappings defined in equations \eqref{eq:a} and \eqref{eq:b}. For any $n\in \Sigma_N,$ the functions $S_n \in \mathcal{HC}^1(I,\mathcal{K}(\mathbb{R}))$ is a compact interval-valued in $I.$ The functions $Q_n:I\times \mathcal{K}(\mathbb{R})\rightarrow\mathcal{K}(\mathbb{R})$ is compact interval-valued continuous and Rakotch contraction in second variable with $\phi_n(\beta t)=\beta\phi_n(t),$ for all $\beta,t>0$ and
\begin{align*}
    \mathfrak{H}&(Q_n(x,Y),Q_n(x_*,Y))\le q_n|x-x_*|\mathfrak{H}(Y,\{0\}),\\
 &Q_n(x,Y)\oplus Q_n(x,Y_*)=Q_n(x,Y\oplus Y_*),
\end{align*}
where 
$$\Bigg\{\Big(1+\frac{N}{\min_{n\in \Sigma_N}|a_n|}\Big)\frac{\max_{n\in \Sigma_N}\phi_n(t)}{t}+ N\frac{\max_{n\in \Sigma_N}q_n}{\min_{n\in \Sigma_N}|a_n|}\Bigg\}<1$$ for all $t>0.$ 
Then there exists a probability vector $p=(p_{n})_{n\in \Sigma_N}$ such that the associated invariant measure $\mu_p$ generated by IFS $\mathcal{J}$ is equivalent to the one-dimensional  Hausdorff measure supported on $\mathcal{G}_*(f).$, i.e.,
    $H^1\vert_{\mathcal{G}_*(f)}.$
\end{theorem}
 \begin{proof}
 Firstly, observe that the IFS $\{I ; \ L_n:n \in \Sigma_N\}$ consists of similarity mappings with similarity ratio $|a_n|,$ and satisfies open set condition. Using \cite[Theorem $2.2$]{Moran}, there exists a probability vector $p^*=(|a_n|)_{n \in \Sigma_N}$ such that the associated invariant measure $\nu_{p^*}$ generated by the IFS $\{I ; \ L_n:n \in \Sigma_N\}$ satisfies
 $$\nu_{p^*}=H^{1}\vert_{I},$$
 where $H^{1}\vert_{I }$ is the one-dimensional Hausdorff measure supported on $I .$ Let $\mu_{p^*}$ be the invariant measure corresponding to $p^*$ generated by IFS $\mathcal{J}.$ Let $B$ be a Borel subset of $\mathcal{G}_*(f).$ Exerting Lemma \ref{homeo}, we get
 \begin{equation}\label{1234}
   \mu_{p^*}(B)=\nu_{p^*}(P^{-1}(B))=H^{1}\vert_{I }(P^{-1}(B)).
 \end{equation}
 With reference to Theorem \ref{holder1}, the fractal function $f\in\mathcal{HC}^1(I,\mathcal{K}(\mathbb{R}))$  with constant $l,$ say. For $x,y\in I,$
\begin{align*}
    |x-y|\le d(P(x),&P(y))=d((x,f(x)),(y,f(y)))=|x-y|+\mathfrak{H}(f(x),f(y))\\&\le|x-y|+l|x-y|=(1+l)|x-y|,
\end{align*}
 claiming that $P$ is a  bi-Lipschitz function. From \cite[Proposition 2.2]{Fal}, for any Borel subset $B$ of $\mathcal{G}_*(f)$, we have
$$H^{1}(P^{-1}(B))\leq H^{1}(B)\leq (1+l)H^{1}(P^{-1}(B)).$$
Therefore, from equation \eqref{1234}, we obtain
$$\mu_{p^*}(B) \leq  H^{1}(B) \leq  (1+l)\mu_{p^*}(B).$$
This implies that $\mu_{p^*}$ is absolutely continuous with respect to $H^1\vert_{\mathcal{G}_*(f)}$ and vice-versa, claiming the assertion.
\end{proof}

\section{Dimension Results}\label{sec4}
This section is dedicated to the dimensional estimation of the distance set and difference set of the graph of SVFs. In the sequel, we give some remarkable comments on the celebrated distance set conjecture given by Falconer \cite{Fal1}. Further, the decomposition of a continuous compact convex SVFs is exhibited, preserving the Hausdorff dimension.

\begin{proposition}
    Let $A,B\in\mathcal{K}(\mathbb{R}).$ Then 
    $$\max\{\dim_{H}A,\dim_{H}B\}\le\dim_{H}\Lambda(A,B).$$
\end{proposition}
\begin{proof}
 Define a  function $\omega:\Lambda(A,B)\to A$ as 
 $$\omega(a,b)=a, ~~\forall~(a,b)\in \Lambda(A,B).$$
Then $\omega$ is a Lipschitz and onto map, as
 \begin{align*}
     |\omega(a,b)-\omega(a_*,b_*)|=|a-a_*|\le\|(a,b)-(a_*,b_*)\|_2.
 \end{align*}
 For any $a\in A,$ note that $\Lambda_B(a)=\{b\in B: D(a,B)=|a-b|\}.$ Define a map $\Phi:B\to\mathbb{R}$ as 
 $$\Phi(b)=|b-a|, ~ ~\forall~ b\in B.$$
Since $\Phi$ is continuous on a compact set $B,$ using the extreme value theorem, there exists $b_*\in B$ such that 
 $$\Phi(b_*)=|b_*-a|=\min_{b\in B}|b-a|.$$
 This gives $b_*\in\Lambda_B(a)$ and $(a,b_*)\in \Lambda(A,B),$ shows that $\omega$ is onto.  In view of \cite[Corollary 2.4]{Fal}, we get 
 $$\dim_HA\le\dim_{H}\Lambda(A,B).$$
 On the same lines, one can show that $\dim_HB\le\dim_{H}\Lambda(A,B),$ establishing the claim.
\end{proof}
Let $D\subseteq \mathbb{R}$ and $f:D\to\mathcal{K}(\mathbb{R})$ be any SVF. Then the graph of $f$ is defined as
\begin{align*}
    &\mathcal{G}(f):=\{(x,y)\in D\times\mathbb{R}:x\in D, y\in f(x)\},\\
    &\mathcal{G}_*(f):=\{(x,f(x))\in D\times\mathcal{K}(\mathbb{R}):x\in D\}.
\end{align*}
It is noted that $\mathcal{G}_*(f)\subset D\times \mathcal{K}(\mathbb{R}),$ so the metric on it is given as $$d((x,f(x)),(y,f(y)))=|x-y|+\mathfrak{H}(f(x),f(y)),$$
for all $(x,f(x))$ and $(y,f(y))\in \mathcal{G}_*(f).$ The distance sets are, respectively, defined as follows:
\begin{align*}
    &\Delta(\mathcal{G}(f)):=\{|x-x_*|+|y-y_*|:x,x_*\in D, y\in f(x), y_*\in f(x_*)\},\\
    &\Delta(\mathcal{G}_*(f)):=\{|x-x_*|+\mathfrak{H}(f(x),f(x_*)):x,x_*\in D\}.
\end{align*}
Subsequently, the difference sets are defined as 
\begin{align*}
    &\Gamma(\mathcal{G}(f)):=\{(x-x_*,y-y_*):x,x_*\in D, y\in f(x), y_*\in f(x_*)\},\\
    &\Gamma(\mathcal{G}_*(f)):=\{(x-x_*, f(x)\oplus(-f(x_*))):x,x_*\in D\}.
\end{align*}
It is obvious to observe that $\Gamma(\mathcal{G}(f))\subset \mathbb{R}\times\mathbb{R}$ and $\Gamma(\mathcal{G}_*(f))\subset \mathbb{R}\times\mathcal{K}(\mathbb{R}).$ Verma and Priyadarshi \cite{VP12} investigated the Hausdorff dimension of the distance set and difference set for the graphs of real-valued continuous functions on the unit interval and showed that the celebrated ``distance-set conjecture'' \cite{Fal1} is valid for a graph of a continuous function. Motivated by the fact, our upcoming result comments on the Hausdorff dimension of the distance set and difference set for the graph of continuous SVFs on the compact interval of $\mathbb{R}.$
\begin{remark}
  Observe that for every set $A\subset\mathbb{R}^2$ (compact), there exists a set $A_*=\{x:(x,y)\in A\}\subset \mathbb{R}$ (bounded) and a SVF $f_A: A_*\rightarrow\mathcal{K}(\mathbb{R})$ defined as 
  $$f_A(x)=\{y\in\mathbb{R}: (x,y)\in A\}.$$
  It is noted that $\mathcal{G}(f_A)=A.$ So, solving distance set conjecture for $\mathbb{R}^2$ (also in higher dimensional Euclidean spaces $\mathbb{R}^n$) is the same as solving the problem for SVFs.
\end{remark}

The upcoming theorem verifies the identification of Falconer's distance conjecture in $\mathbb{R}^2$ for the SVFs.
\begin{theorem}
    Let $f:I\rightarrow\mathcal{K}(\mathbb{R})$ be a continuous SVF such that $\dim_H(\mathcal{G}(f))=\dim_P(\mathcal{G}(f))>1,$  then $\dim_H\Delta(\mathcal{G}(f))=1.$ 
\end{theorem}
\begin{proof}
    It is noted that $\mathcal{G}(f)$ is a Borel subset of $\mathbb{R}^2.$ Consequently, under given assumptions, exerting Theorem $1.1$ of \cite{Shmerkin}, there exists a $(x_0,y_0)\in \mathcal{G}(f)$ such that $$\dim_H\Delta_{(x_0,y_0)}(\mathcal{G}(f))=1,$$
    where $\Delta_{(x_0,y_0)}(\mathcal{G}(f))=\{|x-x_0|+|y-y_0|:x\in I,y\in f(x)\}\subseteq\Delta(\mathcal{G}(f)).$ Since the Hausdorff dimension follows monotonicity, concluding the claim.
\end{proof}
\begin{theorem}
     Let $f:I\rightarrow\mathcal{K}(\mathbb{R})$ be a continuous SVF such that $\dim_A(\mathcal{G}(f))>1,$ then $\dim_A\Delta(\mathcal{G}(f))=1.$ 
\end{theorem}
\begin{proof}
    The proof directly follows from \cite[Theorem 2.3]{Fraser1}.
\end{proof}
\begin{proposition}
    Let $I$ be a compact interval in $\mathbb{R}$ and $f:I\to\mathcal{K}(\mathbb{R})$ be any continuous function. Then the associated distance set $\Delta(\mathcal{G}_*(f))$ of the graph is an interval, and $\dim_H\Delta(\mathcal{G}_*(f))=\underline{\dim}_B\Delta(\mathcal{G}_*(f))=\overline{\dim}_B\Delta(\mathcal{G}_*(f))=1.$
\end{proposition}
\begin{proof}
   Consider a map $\Phi: I\times I\to\mathbb{R}$ as 
   $$\Phi(x,x_*)=|x-x_*|+\mathfrak{H}(f(x),f(x_*)).$$
   Since $f$ is continuous, as a result $\Phi$ is continuous on $I\times I$ and $\Phi(I\times I)=\Delta(\mathcal{G}_*(f)).$ Under continuous function the image of a connected set is connected, $\Delta(\mathcal{G}_*(f))$ is a connected subset of $\mathbb{R}$ containing atleast two distinct points of $\mathbb{R}.$ Thus, $\Delta(\mathcal{G}_*(f))$ is an interval and $\dim_H\Delta(\mathcal{G}_*(f))=\underline{\dim}_B\Delta(\mathcal{G}_*(f))=\overline{\dim}_B\Delta(\mathcal{G}_*(f))=1.$
\end{proof}
\begin{proposition}
    Let $D\subseteq\mathbb{R}$ and $f:D\to\mathcal{K}(\mathbb{R})$ be any bounded function. Then 
    \begin{align*}
        \dim_H\Gamma(\mathcal{G}(f))\le\min\{2, \dim_H\mathcal{G}(f)+\overline{\dim}_B\mathcal{G}(f)\}.
    \end{align*}
\end{proposition}
\begin{proof}
    Consider a function $\Psi:\mathcal{G}(f)\times \mathcal{G}(f)\rightarrow \Gamma(\mathcal{G}(f))$ as 
    $$\Psi((x,y),(x_*,y_*))=(x-x_*,y-y_*),$$
    for all $(x,y), (x_*,y_*)\in \mathcal{G}(f).$ It is straightforward to observe that $\Psi$ is an onto function. Since $\Gamma(\mathcal{G}(f))\subset\mathbb{R}^2,$ we get $\dim_H\Gamma(\mathcal{G}(f))\le 2.$ Notice that 
    \begin{align*}
    \|\Psi((x,y)&,(x_*,y_*))-\Psi((\overline{x},\overline{y}),(\overline{x}_*,\overline{y}_*))\|_2\\
    &=\|(x-x_*,y-y_*)-(\overline{x}-\overline{x}_*,\overline{y}-\overline{y}_*)\|_2 \\
    &\le \|(x-x_*,y-y_*)\|_2+\|(\overline{x}-\overline{x}_*,\overline{y}-\overline{y}_*)\|_2 \\
    &\le2\|(x-x_*,y-y_*)-(\overline{x}-\overline{x}_*,\overline{y}-\overline{y}_*)\|_2,
    \end{align*}
    showing that $\Psi$ is a Lipschitz map. Consequently, we get 
    \begin{align*}
       \dim_H\Gamma(\mathcal{G}(f))\le\dim_H(\mathcal{G}(f)\times \mathcal{G}(f)). 
    \end{align*}
    Also, $\mathcal{G}(f)\subset\mathbb{R}^2.$ By exerting \cite[Product formula 7.3]{Fal}, we obtain $\dim_H(\mathcal{G}(f)\times\mathcal{G}(f))\le \dim_H\mathcal{G}(f)+\overline{\dim}_B\mathcal{G}(f).$ This completes the proof.
\end{proof}
\begin{remark}
    The previous result can also be obtained for the upper box dimension. Let $f:D\to\mathcal{K}(\mathbb{R})$ be any bounded function. Then 
    \begin{align*}
        \overline{\dim}_B\Gamma(\mathcal{G}(f))\le\min\{2, \overline{\dim}_B\mathcal{G}(f)+\overline{\dim}_B\mathcal{G}(f)\}.
    \end{align*}
\end{remark}
Mauldin and Williams \cite{Mauldin}, and later Verma and Massopust \cite{VM}, thoroughly investigated the dimensional aspects of the sum of real-valued continuous functions. Our next theorem highlights the dimension result on the metric-sum of SVFs and offers a useful technique to simplify computations for complicated SVFs.  
\begin{theorem}
  Let $f,g\in \mathcal{C}(I,\mathcal{K}(\mathbb{R}))$ such that $f$ is Lipschitz map with Lipschitz constant $L.$ Then 
  \begin{align*}
      \dim_H\mathcal{G}_*(f\oplus g)=\dim_H\mathcal{G}_*(g).
  \end{align*}
\end{theorem}
\begin{proof}
    Consider a map $T_f:\mathcal{G}_*(g)\rightarrow\mathcal{G}_*(f\oplus g)$ defined as
    $$T_f(x,g(x)):=(x,f(x)\oplus g(x)),$$
    where $x\in I.$ Clearly, $T_f$ is onto. In view of Lemma \ref{11}, and $f$ is Lipschitz map, notice that
    \begin{align*}
        d(T_f(x,g(x)),T_f(y,g(y)))&=d((x,f(x)\oplus g(x)),(y,f(y)\oplus g(y)))\\
        &=|x-y|+\mathfrak{H}(f(x)\oplus g(x),f(y)\oplus g(y))\\
        &\le|x-y|+\mathfrak{H}(f(x)\oplus g(x),f(y)\oplus g(x))+\\
        &\quad\mathfrak{H}(f(y)\oplus g(x),f(y)\oplus g(y))\\
         &=|x-y|+\mathfrak{H}(f(x),f(y))+\mathfrak{H}(g(x),g(y))\\
          &\le|x-y|+L|x-y|+\mathfrak{H}(g(x),g(y))\\
          &\le(1+L)\{|x-y|+\mathfrak{H}(g(x),g(y))\}\\
          &=(1+L)d((x,g(x)),(y,g(y))).
    \end{align*}
    Again, consider
    \begin{align*}
       &d(T_f(x,g(x)),T_f(y,g(y)))\\ 
        =&|x-y|+\mathfrak{H}(f(x)\oplus g(x),f(y)\oplus g(y))\\
        \ge&\frac{1}{(1+L)}\Big\{(1+L)|x-y|+\mathfrak{H}(f(x)\oplus g(x),f(y)\oplus g(y))\Big\}\\
        \ge&\frac{1}{(1+L)}\Big\{|x-y|+\mathfrak{H}(f(x),f(y))+\mathfrak{H}(f(x)\oplus g(x),f(y)\oplus g(y))\Big\}\\
        \ge&\frac{1}{(1+L)}\Big\{|x-y|+\mathfrak{H}(f(x)\oplus g(y), f(y)\oplus g(y))+\mathfrak{H}(f(x)\oplus g(x),f(y)\oplus g(y))\Big\}\\
        \ge&\frac{1}{(1+L)}\Big\{|x-y|+\mathfrak{H}(f(x)\oplus g(y),f(x)\oplus g(x))\Big\}\\
        =&\frac{1}{(1+L)}\Big\{|x-y|+\mathfrak{H}(g(y),g(x))\Big\}\\
        =&\frac{1}{(1+L)}d((x,g(x)),(y,g(y))).
    \end{align*}
    This verifies that $T_f$ is a bi-Lipschitz function. Since the Hausdorff dimension is invariant under bi-Lipschitz mapping (for details, \cite{Fal}), the proof is concluded.
\end{proof}
\begin{remark}
    Since upper and lower box dimensions are also bi-Lipschitz invariant, as a result
    \begin{align*}
        &\overline{\dim}_B\mathcal{G}_*(f\oplus g)=\overline{\dim}_B\mathcal{G}_*(g),\\
        &\underline{\dim}_B\mathcal{G}_*(f\oplus g)=\underline{\dim}_B\mathcal{G}_*(g).
    \end{align*}
\end{remark}
\begin{corollary}\label{coro1}
    If $f\in \mathcal{C}(I,\mathcal{K}(\mathbb{R}))$ is a Lipschitz function. Then $\dim_H\mathcal{G}_*(f)=1.$ 
\end{corollary}
\begin{proof}
    Choose the continuous function $g:I\to\mathcal{K}(\mathbb{R})$ as $g(x)=\{0\},~ \forall ~x\in I.$ Then, $\dim_H\mathcal{G}_*(g)=1$ and $f\oplus g=f.$ Consequently, from the previous theorem 
    $$\dim_H\mathcal{G}_*(f\oplus g)=\dim_H\mathcal{G}_*(f)=\dim_H\mathcal{G}_*(g)=1,$$
    concluding the result.
\end{proof}
\begin{lemma}\label{poly}
Let $p\in  \mathcal{C}(I,\mathcal{K}(\mathbb{R})) $ be any polynomial of degree $n$ as  
$$p(x)=A_0\oplus A_1x\oplus A_2x^2\oplus \cdots\oplus A_nx^n,$$
where $x\in I,$ and $A_0,A_1,\ldots, A_n\in \mathcal{K}_c(\mathbb{R}).$ Then $p$ is Lipschitz.
\end{lemma}
\begin{proof}
    Let $x,y\in I.$ With the reference to Corollary \ref{cor1}, we get
\begin{align*}
    &\mathfrak{H}(p(x),p(y))\\
    &=\mathfrak{H}(A_0\oplus A_1x\oplus A_2x^2\oplus \cdots\oplus A_nx^n,~A_0\oplus A_1y\oplus A_2y^2\oplus \cdots\oplus A_ny^n)\\
    &\le\mathfrak{H}(A_1x,A_1y)+\mathfrak{H}(A_2x^2,A_2y^2)+\cdots+ \mathfrak{H}(A_nx^n,A_ny^n).
\end{align*}
In view of Lemma \ref{12}, and $x^i$ is Lipschitz on $I$ with Lipschitz constant $k_i\in\mathbb{R},$ we obtain
\begin{align*}
\mathfrak{H}(p(x),p(y))
    &\le |A_1||x-y|+|A_2||x^2-y^2|+\cdots+|A_n||x^n-y^n|\\
    &\le |A_1||x-y|+|A_2|k_2|x-y|+\cdots+|A_n|k_n|x-y|\\
    &\le \max\{|A_1|,k_2|A_2|,\ldots, k_n|A_n|\}|x-y|,
\end{align*}
claiming the proof.
\end{proof}
\begin{theorem}\label{th34}
    Every polynomial in $\mathcal{C}(I,\mathcal{K}(\mathbb{R}))$ with coefficients in $\mathcal{K}_c(\mathbb{R})$ has the Hausdorff dimension equal to $1.$
\end{theorem}
\begin{proof}
    The proof directly follows from Corollary \ref{coro1} and Lemma \ref{poly}. Hence, we omit it.
\end{proof}
\begin{lemma}\label{LM413}
Let $\mathcal{C}(I,\mathcal{K}_c(\mathbb{R}))=\{f:I\rightarrow\mathcal{K}_c(\mathbb{R}): f\text{ is continuous}\}.$ Then the set of polynomials with coefficients in $\mathcal{K}_c(\mathbb{R})$ is dense in $\mathcal{C}(I,\mathcal{K}_c(\mathbb{R})).$   
\end{lemma}
\begin{proof}
 In view of \cite[Corollary 6.5]{Dyn3}, required proof is established.   
\end{proof}
Mauldin and Williams \cite{Mauldin} initiated the study of additive decomposition based on the Hausdorff dimension of a real-valued continuous function defined on a compact interval. We now proceed with a sequence of lemmas and theorems to establish the decomposition with respect to metric-sum of a continuous SVF defined on a compact interval of $\mathbb{R}.$ This technique will serve a method to deal with complicated SVFs. Before discussing our next lemma, notice that if $A,B,C\in \mathcal{K}(\mathbb{R})$ such that $A\oplus B=A\oplus C,$ then $B=C.$ For $b\in B,$ there exists $a\in A$ such that $(a,b)\in \Lambda(A,B)$ and $a+b\in A\oplus B.$ This gives $a+b\in A\oplus C$ and $(a,b)\in \Lambda(A,C).$ So, $b\in C$ and $B\subseteq C.$ On the same lines, we get $C\subseteq B,$ concluding that $B=C.$ 
\begin{lemma}\label{lemma410}
Let $\mathcal{C}_0(I,\mathcal{K}_c(\mathbb{R}))$ be a $G_\delta$-dense subset in $\mathcal{C}(I,\mathcal{K}_c(\mathbb{R})).$ Then for any $f\in \mathcal{C}(I,\mathcal{K}_c(\mathbb{R})),$ the set $\mathcal{C}_0(I,\mathcal{K}_c(\mathbb{R}))\oplus f$ defined as 
    $$\mathcal{C}_0(I,\mathcal{K}_c(\mathbb{R}))\oplus f=\{g\oplus f: g\in \mathcal{C}_0(I,\mathcal{K}_c(\mathbb{R}))\}$$
is $G_\delta$-dense.
\end{lemma}
\begin{proof}
Since $\mathcal{C}_0(I,\mathcal{K}_c(\mathbb{R}))$ is given to be $G_\delta$-dense, there exist $U_i\ (i\in \mathbb{N})$ open and dense sets in $\mathcal{C}(I,\mathcal{K}_c(\mathbb{R}))$ such that 
$$\mathcal{C}_0(I,\mathcal{K}_c(\mathbb{R}))=\bigcap_{i=1}^\infty U_i.$$
Our initial target is to show $\mathcal{C}_0(I,\mathcal{K}_c(\mathbb{R}))\oplus f=\bigcap_{i=1}^\infty (U_i\oplus f).$
\begin{align*}
    g\oplus f\in \mathcal{C}_0(I,\mathcal{K}_c(\mathbb{R}))\oplus f\implies g\in \mathcal{C}_0(I,\mathcal{K}_c(\mathbb{R}))\implies g\in U_i, \forall i\implies g\oplus f\in U_i\oplus f,\forall ~i \in \mathbb{N}.
\end{align*}
Therefore, $\mathcal{C}_0(I,\mathcal{K}_c(\mathbb{R}))\oplus f\subseteq \bigcap_{i=1}^\infty (U_i\oplus f).$ Again
\begin{align*}
    g'\oplus f\in\bigcap_{i=1}^\infty (U_i\oplus f)\implies g'\oplus f\in U_i\oplus f, \forall ~i \in \mathbb{N}\implies g'\in U_i, \forall ~i \in \mathbb{N}\implies g'\in \mathcal{C}_0(I,\mathcal{K}_c(\mathbb{R})).
\end{align*}
This gives $\bigcap_{i=1}^\infty (U_i\oplus f)\subseteq \mathcal{C}_0(I,\mathcal{K}_c(\mathbb{R}))\oplus f.$
Also $U_i$ is open subset of $\mathcal{C}(I,\mathcal{K}_c(\mathbb{R})),$ for any $h\in U_i,$ there exists $\delta>0$ such that $N_\delta(h)\subseteq  U_i,$ where $N_\delta(h)=\{g\in\mathcal{C}(I,\mathcal{K}_c(\mathbb{R})): d_C(g,h)<\delta \}.$ Consider $$d_C(g\oplus f,h\oplus f)=\sup_{x\in I}\mathfrak{H}(g(x)\oplus f(x),h(x)\oplus f(x))=\sup_{x\in I}\mathfrak{H}(g(x),h(x))=d_C(g,h).$$ 
So, if $g\in N_\delta(h)$ implies $g\oplus f\in N_\delta(h\oplus f).$ Consider $\phi\in N_\delta(h\oplus f),$ then 
\begin{align*}
   d_C(\phi\oplus(-f),h)&=\sup_{x\in I}\mathfrak{H}\Big(\phi(x)\oplus(-f(x)),h(x)\Big)\\ 
   &=\sup_{x\in I}\mathfrak{H}\Big(\phi(x),h(x)\oplus f(x)\Big)\\ 
   &=d_C(\phi,h\oplus f).
\end{align*}
This verifies that $\phi\oplus(-f)\in N_\delta(h).$ As a result, $N_\delta(h\oplus f)=N_\delta(h)\oplus f,$ and $U_i\oplus f$ is an open subset of $\mathcal{C}(I,\mathcal{K}_c(\mathbb{R})).$ It is remaining to show that $U_i\oplus f$ is dense. Define a map $T:\mathcal{C}(I,\mathcal{K}_c(\mathbb{R}))\rightarrow\mathcal{C}(I,\mathcal{K}_c(\mathbb{R}))$ as $T(g)=g\oplus f.$ It is straightforward to observe that $T$ is bijective and continuous on $\mathcal{C}(I,\mathcal{K}_c(\mathbb{R})).$ Then
$$T(\overline{U_i})\subseteq \overline{T(U_i)}.$$
Also, for any $h\in \overline{T(U_i)} $ and $\epsilon>0,$ there exists $h_*\in T(U_i)$ such that $$d_C(h,h_*)<\epsilon.$$
Since $h_*\in T(U_i),$ we get $h'\in U_i$ with $T(h')=h_*.$ As a consequence 
\begin{align*}
    d_C(h,h_*)=d_C(h,h'\oplus f)=d_C(h\oplus(-f),h')<\epsilon.
\end{align*}
So, $h\oplus(-f)\in \overline{U_i}$ and $h\in T(\overline{U_i}).$ Hence, $U_i\oplus f$ is dense in $\mathcal{C}(I,\mathcal{K}_c(\mathbb{R})).$ This completes the proof.
\end{proof}
In the upcoming theorem, the decomposition of a continuous SVF is exhibited. 
\begin{theorem}\label{thm411}
The set $\mathcal{D}_1=\{f\in \mathcal{C}(I,\mathcal{K}_c(\mathbb{R})): \dim_H\mathcal{G}_*(f)=1\}$ is a $G_\delta$-dense subset of $\mathcal{C}(I,\mathcal{K}_c(\mathbb{R})).$ Furthermore, if $f\in\mathcal{C}(I,\mathcal{K}_c(\mathbb{R})),$ then $g=h\oplus f,$ for some $g,h\in \mathcal{D}_1.$    
\end{theorem}
\begin{proof}
For $\beta>1,$ consider a set $\mathcal{C}_\beta=\{f\in \mathcal{C}(I,\mathcal{K}_c(\mathbb{R})): ~H^\beta\mathcal{G}_*(f)=0\}.$ Then $\mathcal{C}_\beta$ is a $G_\delta$-subset of $\mathcal{C}(I,\mathcal{K}_c(\mathbb{R})).$ 
Using Theorem \ref{th34} and the definition of $\beta$-dimensional Hausdorff measure, it follows that the set of polynomials with coefficients as a compact interval of $\mathbb{R}$, with respect to metric-sum, is contained in $ \mathcal{C}_\beta.$ In view of Lemma \ref{LM413}, $\mathcal{C}_\beta$ is dense in $\mathcal{C}(I,\mathcal{K}_c(\mathbb{R})).$ Also, $\mathcal{D}_1=\bigcap_{n\in\mathbb{N}}\mathcal{C}_{1+\frac{1}{n}}.$ It follows from Lemma \ref{lemma410}, for any $f\in \mathcal{C}(I,\mathcal{K}_c(\mathbb{R})),$ $\mathcal{C}_\beta\oplus f$ is $G_\delta$-dense. Thus, $\mathcal{D}_1$ and $\mathcal{D}_1\oplus f$ are $G_\delta$-dense subsets of $\mathcal{C}(I,\mathcal{K}_c(\mathbb{R})).$ Now, $\mathcal{D}_1\cap\mathcal{D}_1\oplus f\neq\emptyset,$ otherwise contradiction occurs using the Baire category theorem. Therefore, there exist $g,h\in \mathcal{D}_1$ with $g=h\oplus f$, this completes the proof.
\end{proof}
\begin{corollary}
If $f\in\mathcal{C}(I,\mathcal{K}_c(\mathbb{R})),$ then there exist $g,h\in\mathcal{D}_1$ such that $f=g\oplus h.$      
\end{corollary}
\begin{lemma}
For any $\beta>1,$ the set $\mathcal{D}_\beta=\{f\in \mathcal{C}(I,\mathcal{K}_c(\mathbb{R})): \dim_H\mathcal{G}_*(f)=\beta\}$ is dense in $\mathcal{C}(I,\mathcal{K}_c(\mathbb{R})).$ 
    \end{lemma}
\begin{proof}
Let $g\in\mathcal{C}(I,\mathcal{K}_c(\mathbb{R})).$ So, for any $\epsilon>0,$ there exists a polynomial $p\in\mathcal{C}(I,\mathcal{K}_c(\mathbb{R}))$ such that 
$$d_C(g,p)<\frac{\epsilon}{2}.$$
For any $f\in\mathcal{D}_\beta,$ consider $h=p\oplus \frac{\epsilon}{2\|f\|_C}f.$ Since $p$ is a Lipschitz function, we get
$$\dim_H\mathcal{G}_*(h)=\dim_H\mathcal{G}_*(f)=\beta.$$
Observe that 
\begin{align*}
   d_C(p,h)&=\sup_{x\in I}\mathfrak{H}(p(x),h(x))\\&=\sup_{x\in I}\mathfrak{H}\Big(p(x),p(x)\oplus\frac{\epsilon}{2\|f\|_C}f(x)\Big)\\
   &=\sup_{x\in I}\mathfrak{H}\Big(\{0\},\frac{\epsilon}{2\|f\|_C}f(x)\Big)\\&\le \frac{\epsilon}{2\|f\|_C}\sup_{x\in I}\mathfrak{H}(\{0\},f(x))=\frac{\epsilon}{2}.
\end{align*}
In view of the triangle inequality, we have
$$d_C(g,h)\le d_C(g,p)+d_C(p,h)<\frac{\epsilon}{2}+\frac{\epsilon}{2}=\epsilon,$$ establishing the claim. 
    
\end{proof}
\begin{lemma}\label{l413}
    Let $X$ be a compact subset of $I.$ Then any continuous function $f:X\to\mathcal{K}_c(\mathbb{R})$ can be extended continuously to $\tilde{f}:I\to\mathcal{K}_c(\mathbb{R})$ such that 
$$\dim_H\mathcal{G}_*(\tilde{f})=\max\{1,\dim_H\mathcal{G}_*(f)\}.$$
\end{lemma}
\begin{proof}
Let $I=[a,b]$ and $X$ be a compact subset of $I,$ then the following cases are possible:
\begin{align*}
    &a,b\in X ~\text{ gives }~ I\setminus X=\cup_{i\in\mathbb{N}}(a_i,b_i),\\
    &a,b\notin X ~\text{ gives }~ I\setminus X=\cup_{i\in\mathbb{N}}(a_i,b_i)\cup\{a,b\},\\
    &a\in X, b\notin X ~\text{ gives }~ I\setminus X=\cup_{i\in\mathbb{N}}(a_i,b_i)\cup\{b\},\\
    &b\in X, a\notin X ~\text{ gives }~ I\setminus X=\cup_{i\in\mathbb{N}}(a_i,b_i)\cup\{a\}.    
\end{align*}  
    In the first case, define a function $\tilde{f}: I\rightarrow\mathcal{K}(\mathbb{R})$ (using metric polynomial interpolant \cite{Dyn3}) as follows: 
    \begin{align}\label{e413}
      \tilde{f}(x):= \left\{\begin{array}{rcl}
    f(x), & \mbox{if} ~~ x\in X, \\ \frac{x-a_i}{b_i-a_i}f(b_i)\oplus\frac{x-b_i}{a_i-b_i}f(a_i), & \mbox{if} \  x\notin X, ~x\in (a_i,b_i).\end{array}\right. \  
    \end{align}
    Clearly, $\tilde{f}(a_i)=f(a_i)$ and $\tilde{f}(b_i)=f(b_i).$ It is noted that the Hausdorff dimension follows countable stability and by Theorem \ref{th34}, we get
    \begin{align*}
        \dim_H\mathcal{G}_*(\tilde{f})&=\sup_{i\in\mathbb{N}}\Big\{\dim_H\mathcal{G}_*(\tilde{f}\big|_X),\dim_H\mathcal{G}_*\Big(\tilde{f}\big|_{(a_i,b_i)}\Big)\Big\}\\
        &=\max\{\dim_H\mathcal{G}_*(f),1\}.
    \end{align*}
    Similarly, remaining cases are observed. This concludes the proof.
\end{proof}
\begin{theorem}
    Let $f\in\mathcal{C}(I,\mathcal{K}_c(\mathbb{R}))$ and $\dim_H\mathcal{G}_*(f) \ge \beta>1.$ Then there exist two functions (say $g_*$ and $h_*$) in $\mathcal{D}_\beta$ such that $$f=g_*\oplus h_*.$$
\end{theorem}
\begin{proof}
  In view of Theorem \ref{thm411}, there exist $g,h\in \mathcal{D}_1$ such that 
  $$f=g\oplus h.$$
  Also, $\dim_H\mathcal{G}_*(f) \ge \beta,$ from \cite[Theorem 4.10]{Fal}, there exists a compact set $A\subseteq I$ such that
  $$\dim_H\mathcal{G}_*\Big(f\big|_A\Big)=\beta.$$
Consider the restriction maps on $A$ as $\bold{g}=g\big|_A$ and $\bold{h}=h\big|_A.$ 
With reference to Lemma \ref{l413}, $\bold{g}$ and $\bold{h}$ can be continuously extended to $\bold{\tilde{g}}$ and $\bold{\tilde{h}},$ respectively. Define functions $g_*, h_*:I\to\mathcal{K}(\mathbb{R})$ as $$g_*=h\oplus\frac{1}{2}\bold{\tilde{g}}\oplus\Big(\frac{-1}{2}\bold{\tilde{h}}\Big)~ \text{  and  }~h_*=g\oplus\frac{1}{2}\bold{\tilde{h}}\oplus\Big(\frac{-1}{2}\bold{\tilde{g}}\Big).$$ 
Under given assumptions, metric-sum is associative, so $g_*\oplus h_*=g\oplus h=f.$ Notice that $g_*=\frac{1}{2}f$ on $A,$ so $\dim_H\mathcal{G}_*\Big(g_*\big|_A\Big)=\dim_H\mathcal{G}_*\Big(\frac{1}{2}f\big|_A\Big)=\dim_H\mathcal{G}_*\Big(f\big|_A\Big)=\beta.$ From equation \eqref{e413}, we see that the function $\frac{1}{2}\bold{\tilde{g}}\oplus\Big(\frac{-1}{2}\bold{\tilde{h}}\Big)$ is a piecewise linear polynomial on $[0,1]\setminus A.$ Also, from countable stability of Hausdorff dimension
\begin{align*}
  \dim_H\mathcal{G}_*(g_*\big|_{I\setminus A})=\sup_{i\in\mathbb{N}}&  \dim_H\mathcal{G}_*(g_*\big|_{(a_i,b_i)})=\sup_{i\in\mathbb{N}}  \dim_H\mathcal{G}_*(h\big|_{(a_i,b_i)})\\
  &=\dim_H\mathcal{G}_*(h\big|_{I\setminus A})\le\dim_H\mathcal{G}_*(h)=1.
\end{align*}
Thus,
\begin{align*}
    \beta\le\dim_H\mathcal{G}_*(g_*)&=\max\Big\{\dim_H\mathcal{G}_*(g_*\big|_A),\dim_H\mathcal{G}_*(g_*\big|_{I\setminus A})\Big\}\\
    &\le\max\{\beta,1\}=\beta.\\
\end{align*}
On the same lines, $\dim_H\mathcal{G}_*(h_*)=\beta,$ establishing the claim.
\end{proof}
After showing the existence of decomposition functions, a natural quest may come to mind to construct such functions. For real-valued functions, Wingren \cite{Wing} was the first to construct such decomposition functions whose Hausdorff dimensions are one. Following the approach of \cite{Wing}, we are working on this issue for SVFs in another manuscript. Hopefully, in the future, our next manuscript will answer our query because $1 \le  \dim_H(\mathcal{G}(f)) \le \underline{\dim}_B(\mathcal{G}(f)) = 1$ implies $\dim_H(\mathcal{G}(f))=1$ for every $f \in \mathcal{C}(I,\mathcal{K}(\mathbb{R}))$. 

Now, we note the next interesting result and remark. To the best of our knowledge, we could not find any literature to record these facts.
\begin{theorem}
    $\dim_H(\mathcal{K}(\mathbb{R})) \ge 2$.
\end{theorem}
\begin{proof}
    Let us first begin by defining $\Phi: \mathcal{K}(\mathbb{R}) \to\mathbb{R}^2$ as $\Phi(A)=(x_*,y_*),$ where $x_*=\min_{x \in A} x$ and $y_*=\max_{y \in A} y$. Since $A$ is compact, the map is well-defined. Observe that $\Phi(\mathcal{K}(\mathbb{R}))= \{(x,y) \in \mathbb{R}^2: x\le y\}$ and $\dim_H(\{(x,y) \in \mathbb{R}^2: x\le y\})=2.$ Let $A,B \in \mathcal{K}(\mathbb{R})$ and $\mathbb{R}^2$ be equipped with the Euclidean metric. Using the fact that $|x_* -\overline{x}_*| \le \mathfrak{H}(A,B) $ and $|y_* -\overline{y}_*| \le \mathfrak{H}(A,B) $ where $x_*=\min_{x \in A} x$,  $y_*=\max_{y \in A} y$, $\overline{x}_*=\min_{x \in B} x$ and $\overline{y}_*=\max_{y \in B} y$, we obtain
    $$\|\Phi(A) -\Phi(B)\|_2 = \|(x_*,y_*) -(\overline{x}_*,\overline{y}_*) \|_2 \le \sqrt{2} \mathfrak{H}(A,B),$$
    that is, $\Phi$ is Lipschitz. Using \cite[Corollary 2.4]{Fal}, $\dim_H(\mathcal{K}(\mathbb{R})) \ge \Phi(\mathcal{K}(\mathbb{R}))=2$, this completes the assertion.
\end{proof}
\begin{remark}
    The mapping $\Phi: C_n=\{(x_1,x_2,x_3,\dots,x_n) \in \mathbb{R}^n: x_1 \le x_2\le \dots \le x_n\}  \to \mathcal{K}(\mathbb{R})$ defined by $\Phi(x_1,x_2,x_3,\dots,x_n)=\{x_1,x_2,x_3,\dots,x_n\}$ is one-one. Every finite set is compact and $\dim_H(C_n)=n.$ Also, 
    \begin{align*}
    \mathfrak{H}(\Phi(x_1,x_2,x_3,\dots,x_n),& \Phi(y_1,y_2,y_3,\dots,y_n))\\
    &=\mathfrak{H}(\{x_1,x_2,x_3,\dots,x_n\},\{y_1,y_2,y_3,\dots,y_n\} ) \\&\le \max_{1 \le i\le n} |x_i -y_i|.
    \end{align*}
    So, $\Phi$ is Lipschitz with respect to max-norm on $C_n.$ It follows that $n= \dim_H(C_n) =\dim_H(\Phi(C_n)) \le \dim_H(\mathcal{K}(\mathbb{R}))$. Since  $n \in \mathbb{N}$ was arbitrary, $\dim_H(\mathcal{K}(\mathbb{R})) \to \infty.$
\end{remark}

\subsection*{Acknowledgements}
The first author is financially supported by a Ministry of Education Fellowship at the Indian Institute of Information Technology (IIIT), Allahabad. The second author is supported by the SEED grant project of IIIT Allahabad.

\bibliographystyle{amsplain}

\end{document}